\documentclass{article}
\usepackage[margin = 1in]{geometry}
\usepackage[utf8]{inputenc}
\usepackage[all]{xy}
\usepackage{graphicx}
\usepackage{amssymb,amsmath,mathtools,amscd}
\usepackage[makeroom]{cancel}
\usepackage{amsthm}
\usepackage{mathabx}
\usepackage{epstopdf}
\usepackage{tikz, tikz-cd}
\usepackage{indentfirst}
\usepackage{hyperref}
\hypersetup{
	colorlinks=true,
	linkcolor=black,
	filecolor=black,      
	urlcolor=black,
	citecolor=blue,
}

\newcommand{\End}{{\rm End}}
\newcommand{\CC}{\mathbb{C}}

\newcommand{\QQ}{\mathbb{Q}}
\newcommand{\ZZ}{\mathbb{Z}}
\newcommand{\NN}{\mathbb{N}}
\newcommand{\kk}{\Bbbk}
\newcommand{\KK}{\mathbb{K}}
\newcommand{\FF}{\mathbb{F}}
\newcommand{\bsep}{\beta_{\mathrm{sep}}}
\newcommand{\bfield}{\beta_{\mathrm{field}}}

\newcommand{\Dreg}{D_{\mathrm{reg}}}
\newcommand{\Dspan}{D_{\mathrm{span}}}
\newcommand{\Dirr}{D^\otimes_\mathrm{irr}}
\newcommand{\topdeg}{\operatorname{topdeg}}
\newcommand{\Char}{\operatorname{char}}
\newcommand{\Frac}{\operatorname{Frac}}
\newcommand{\Aut}{\operatorname{Aut}}
\newcommand{\im}{\operatorname{im}}

\newcommand{\bfi}{\mathbf{i}}
\newcommand{\bfj}{\mathbf{j}}
\newcommand{\bfk}{\mathbf{k}}

\newcommand{\Reynolds}{\mathcal{R}}
\newcommand{\Tr}{\operatorname{Tr}}
\newcommand{\Hom}{\operatorname{Hom}}

\newcommand{\Spec}{\operatorname{Spec}}

\newcommand{\Irr}{\operatorname{Irr}}
\newcommand{\Vreg}{V_\mathrm{reg}}
\newcommand{\sslash}{\mathord{/\mkern-6mu/}}
\newcommand{\init}{\operatorname{in}}
\newcommand{\mcM}{\mathcal{M}}
\newcommand{\mcB}{\mathcal{B}}
\newcommand{\Gal}{\operatorname{Gal}}
\newcommand{\Sym}{\operatorname{Sym}}

\newtheorem{theorem}{Theorem}[section]
\newtheorem{lemma}[theorem]{Lemma}
\newtheorem{corollary}[theorem]{Corollary}
\newtheorem{proposition}[theorem]{Proposition}

\newtheorem{observation}[theorem]{Observation}

\newtheorem*{theorem*}{Theorem}
\newtheorem*{proposition*}{Proposition}
\newtheorem*{conjecture*}{Conjecture}

\theoremstyle{definition}
\newtheorem{definition}[theorem]{Definition}
\newtheorem{example}[theorem]{Example}
\newtheorem*{remark}{Remark}
\newtheorem{notation}[theorem]{Notation}

\makeatletter
\newcommand{\subjclass}[2][1991]{
  \let\@oldtitle\@title
  \gdef\@title{\@oldtitle\footnotetext{#1 \emph{Mathematics Subject Classification.} #2}}
}
\newcommand{\keywords}[1]{%
  \let\@@oldtitle\@title%
  \gdef\@title{\@@oldtitle\footnotetext{\emph{Key words and phrases.} #1.}}%
}
\makeatother

\title{Generic orbits, normal bases, and generation degree for fields of rational invariants}
\author{Ben Blum-Smith and Harm Derksen}
\date{\today}


\begin{document}

\maketitle

\begin{abstract}
    For a faithful linear representation $V$ of a finite group $G$ in characteristic not dividing the group order, we show that if the {\em field Noether number} $\bfield$ is the minimum $d$ such that the invariant polynomials of degree $\leq d$ generate the field $\kk(V)^G$ of rational invariants as a field, and the {\em spanning degree} $\Dspan$ is the minimum $d$ such that the polynomials of degree $\leq d$ span the rational function field $\kk(V)$ as a vector space over $\kk(V)^G$, then $\bfield \leq 2\Dspan + 1$, and this is sharp. This generalizes a recent result of Edidin and Katz.
    
    We also study $\Dspan$. We show that it is related to various quantities previously studied in invariant and representation theory. Dropping the hypothesis on the field characteristic, we prove several basic inequalities, including that it is monotonically nondecreasing in $G$,  nonincreasing in $V$, and satisfies $\Dspan \leq |G|-1$. The latter refines a recent result of Koll\'ar and Tiep.
\end{abstract}

\tableofcontents 

\section{Introduction}

Let $G$ be a finite group, $\kk$ a field, and $V$ a finite-dimensional, faithful representation of $G$ over $\kk$. Let $\Dspan$ denote the minimum degree $d$ so that the polynomials on $V$ of degree at most $d$ span the field $\kk(V)$ of all rational functions as a vector space over the field $\kk(V)^G$ of rational $G$-invariants. Let $\bfield$ denote the minimum degree $d$ such that the $G$-invariant polynomials of degree at most $d$ generate the field $\kk(V)^G$ as a field extension of $\kk$. The main objective of this paper is to prove the following sharp inequality relating these quantities:

\begin{theorem}\label{thm:main}
    Assume that the characteristic of $\kk$ does not divide $|G|$. Then
    \[
    \bfield \leq 2\Dspan + 1.
    \]
\end{theorem}

In this introduction, we provide context and motivation for this theorem, mention some ancillary results, and discuss the method of proof.

\subsection{Context and motivation}

The study of $\bfield$ is a recent variant on long-standing research programs in invariant theory. It has an additional motivation coming from signal processing. To elaborate:

The ring $\kk[V]^G$ of $G$-invariant polynomials on $V$ is finitely generated as a $\kk$-algebra, and there is a very old, but still active, research program to get control over the degrees of the generators \cite{noether, schmid1991finite, domokos-hegedus, sezer2002sharpening, fleischmann2000noether, fogarty2001noether, fleischmann2006noethermodular, symonds2011castelnuovo, cziszter-domokos, cziszter2014noether, cziszter2016interplay, cziszter2019lower, gandini2019ideals, hegedHus2019finite, ferraro2021noether, kemper2025schemes}. The minimum degree $d$ such that $\kk[V]^G$ is generated as a $\kk$-algebra by its elements of degree $\leq d$ is called the {\em Noether number}, and denoted $\beta(G,V)$.

There is a younger program to bound the degrees of polynomials that achieve the less ambitious (but still very useful) goal of separating the orbits of $G$ on $V$ \cite{derksen-kemper, domokos2007typical, draisma2008polarization, kemper2009separating, sezer2009constructing, dufresne2009separating, dufresne2009cohen, kohls-kraft, domokos2011helly, dufresne2013finite, kohls2013separating,  dufresne2015separating, domokos, reimers2018separating, reimers2020separating, lopatin2021separating, domokos2022separating, kemper2022separating, 
domokos2024separating, schefler2025separating, 
schefler2023separating,  zhao2025separating, jia2025modular}. The minimum $d$ such that there exists a separating set is called the {\em separating Noether number}, and is written $\bsep(G,V)$.

More recently still, a program is developing to bound the degrees  of the invariant polynomials required to generate the field $\kk(V)^G$ of invariant rational functions as a field extension of $\kk$ \cite{fleischmann2007homomorphisms, hubert-labahn, blum2024degree, blum2024rational, edidin2025orbit, edidin2025generic, reimers2025generic}. As mentioned above, the minimum $d$ such that $\kk(V)^G$ is generated by polynomial invariants of degree $\leq d$ is denoted $\bfield$, or $\bfield(G,V)$ to emphasize the dependence on $G$ and $V$. Field generation is a less stringent demand on a set of invariants than algebra generation; and if $\kk$ is algebraically closed of characteristic zero, it is also less stringent than orbit separation, because generation of the field $\kk(V)^G$ is equivalent under this hypothesis to {\em generic} orbit separation. So this program is a relaxation of both of the above.

The recent interest in $\bfield(G,V)$ is in part motivated by a circle of signal processing problems known as {\em orbit recovery} or {\em (generalized) multi-reference alignment} \cite{abbe2017sample, perry2019sample, bandeira2020optimal, fan2021maximum, bendory2022sparse, abas2022generalized,  bendory2022dihedral, bandeira2023estimation, edidin2024orbit, edidin2025orbit, bendory2025generalized}. These problems concern reconstruction of (the orbit of) a signal from many samples that have been corrupted both by random additive noise and by transformations randomly drawn from a group. An important example is cryo-electron microscopy, where many very noisy images of a molecule are viewed from unknown directions \cite{sigworth2016principles, singer2018mathematics, bendory2020single}. It is shown in \cite{bandeira2023estimation} that, in the high-noise regime, the number of samples required for accurate reconstruction varies as $\sigma^{2d}$, where $\sigma$ is the noise level and $d$ is the minimum degree so that the invariant polynomials of degree $\leq d$ distinguish the orbit of the sample from other orbits. For a generic signal, $\bfield$ provides a bound on this $d$.

While $\Dspan$ is not to our knowledge the subject of previous research attention, it is closely related to several other quantities of interest in invariant and representation theory:

\begin{itemize}
    \item Researchers in invariant theory have investigated the minimum $d$ such that the ring of polynomial functions $\kk[V]$ is generated as a module over the invariant algebra $\kk[V]^G$ by the polynomials of degree $\leq d$. This quantity is known as $\topdeg(G,V)$ or $\beta(\kk[V],\kk[V]^G)$ \cite{schmid1991finite, kohls2014top, cziszter2019lower} (and see also \cite{cziszter2013generalized, cziszter-domokos, cziszter2014noether}, which study a common generalization of $\topdeg(G,V)$ and the classical Noether number $\beta(G,V)$). The notation $\topdeg$ (for ``top degree") refers to the fact that it can be alternatively defined as the maximum nonzero degree in the {\em coinvariant algebra} $\kk[V]_G:= \kk[V]/\kk[V]^G_{+}\kk[V]$, the quotient of the polynomial algebra by the Hilbert ideal (generated by the invariants of positive degree). The definitions are equivalent by the graded Nakayama lemma.
    
    The quantity $\Dspan(G,V)$ is to $\topdeg(G,V)$ as $\bfield(G,V)$ is to the classical Noether number $\beta(G,V)$, i.e., it replaces a quantity related to ring and module structure with the analogous quantity related to field and vector space structure. The study of $\Dspan(G,V)$ was originally proposed to the first-named author by Victor Reiner  \cite[Question~5.11]{blum2024degree} in view of this analogy. One has immediately that $\Dspan(G,V) \leq \topdeg(G,V)$, i.e., $\Dspan$ is a relaxation of $\topdeg(G,V)$.
    
    \item A well-known theorem of Burnside asserts  (when $\kk = \CC$) that every irreducible representation of $G$ occurs as a constituent of some tensor power of $V$. (This has been generalized in various ways \cite{steinberg1962complete, rieffel1967burnside, passman1995burnside, krason-kuhn}; when the field characteristic divides the group order, one interprets the word {\em constituent} as {\em composition factor}.) In 1964, Brauer \cite{brauer1964note} proved a bound on the minimum $d$ such that the tensor powers up to the $d$th are sufficient to collect every irreducible. Call this quantity $\Dirr(G,V)$, or $\Dirr$ for short;\footnote{It is referred to as $c(V)$ in \cite[Definition~2.1]{wehlau2003some}, but we use $\Dirr$ to be consistent with $\Dspan, \Dreg$.} then \cite[Theorem~$1^*$]{brauer1964note} asserts that 
    \[
    \Dirr\leq r-1,
    \]
    where $r$ is the number of distinct values that the character of the representation $V$ takes on $G$. This theorem has been generalized to monoids \cite{steinberg2014burnside}. 
    
    It follows from the Normal Basis Theorem, and will be proven below (Proposition~\ref{prop:c-Dreg-Dspan-topdeg}), that $\Dspan$ is an upper bound on $\Dirr$; furthermore, if $G$ is abelian and $\Char \kk$ does not divide $|G|$, then they are equal.
    
    \item A line of research in representation and invariant theory concerns the structure of the polynomial algebra $\kk[V]$ as a representation of $G$ \cite{bryant1993symmetrical, bryant1995groups, karagueuzian1999module, symonds2000group, karagueuzian2007module, karagueuzian2004module, symonds2007cyclic, symonds2024module}. A finding of this literature is that, as one increases the degree $d$, the $\kk$-space of polynomials $\kk[V]_{\leq d}$ of degree $\leq d$  asymptotically approaches many copies of the regular representation. In \cite{kollar2024simple} the authors study the minimum degree $d$ so that $\kk[V]_{\leq d}$ has the regular representation as a summand for the very first time. Call this number $\Dreg(G,V)$. Then the proof of the main result of \cite{kollar2024simple} shows that $\Dreg(G,V) \leq |G|-1$. This result does not depend on the characteristic of $\kk$.

    It will be shown below that $\Dreg \leq \Dspan$, and that, as with $\Dirr$, they are equal when $G$ is abelian and the characteristic does not divide $|G|$ (see Proposition~\ref{prop:c-Dreg-Dspan-topdeg}). Furthermore, the point of view taken in the present work leads to a short proof that $\Dspan \leq |G|-1$ in arbitrary characteristic, slightly refining the main result of \cite{kollar2024simple}.
\end{itemize}

Thus, Theorem~\ref{thm:main} relates two quantities $\bfield$ and $\Dspan$ that are individually of interest. To demonstrate usefulness, a quick application is given in Section~\ref{sec:applications}.

In at least one respect, $\Dspan(G,V)$ is better-behaved than $\bfield(G,V)$: it satisfies monotonicity properties in both arguments (Theorem~\ref{thm:monotonic-in-the-rep} and Lemma~\ref{lem:monotonic-in-the-group} below). Both of these properties fail for $\bfield$. Thus, one can hope that Theorem~\ref{thm:main} bounds a more unruly quantity with a better-behaved one.

Theorem~\ref{thm:main} also generalizes a recent result of Edidin and Katz \cite{edidin2025orbit}. Based on suggestive results from the signal processing literature mentioned above, it has been suspected (e.g., \cite[Remark~4.3]{bandeira2023estimation}) that $\bfield(G,V)\leq 3$ if $V$ is the regular representation of $G$ over $\CC$. Edidin and Katz showed that if $V$ so much as contains the regular representation, then $\bfield(G,V)\leq 3$.\footnote{Edidin and Katz actually proved that the invariant polynomials of degree $\leq 3$ separate generic orbits of $G$ on $V$, with no hypothesis on the field except that it is infinite. In characteristic zero, $\bfield(G,V) \leq 3$ follows; this will be discussed further in Section~\ref{ex:regular} below.} The hypothesis that $V$ contains the regular representation implies that $\Dspan = 1$. Thus, we recover their conclusion that $\bfield(G,V)\leq 3$ from the special case of Theorem~\ref{thm:main} in which $\Dspan(G,V) = 1$. 

The implication mentioned in the last paragraph, that Edidin and Katz' hypothesis that $V$ contains the regular representation implies that $\Dspan=1$, is argued carefully in Section~\ref{ex:regular}, but let us outline the argument here in an effort to give the reader a sense of why it is plausible. The above-referenced monotonicity of $\Dspan$ in the representation (Theorem~\ref{thm:monotonic-in-the-rep} below) implies we can reduce to the case that $V$ is itself the regular representation. Let $x_1,\dots,x_n$ be  coordinate functions on $V$. Since $V$ is the regular representation, $n=|G|$, and the matrix $(gx_j)_{g,j}$ for $g\in G$ and $j=1,\dots,n$ is a square matrix. By evaluating at a $v\in V$ whose $G$-images $gv$ are linearly independent, we see that this matrix must be nonsingular. It follows that the coordinate functions are linearly independent over the invariant field because a nontrivial linear relation between them with invariant coefficients would also yield a nontrivial element in the kernel of the matrix $(gx_j)_{g,j}$. Dimension counting then yields that the coordinate functions span $\kk(V)$ over the invariant field, thus $\Dspan = 1$.

\subsection{Results}\label{sec:results}

\begin{notation}
    We adopt the following standard notation. The ring of polynomials on $V$ is $\kk[V]$. The subring invariant under the natural $G$-action given by $(gf)(v) := f(g^{-1}v)$ is  $\kk[V]^G$. By $\kk[V]_{\leq d}$ we denote the $\kk$-vector space of polynomials on $V$ of degree less than or equal to $d$, while $\kk[V]_d$ is the $\kk$-vector space of {\em homogeneous} polynomials of degree exactly $d$. The field of rational functions on $V$ with coefficients in $\kk$ is $\kk(V)$, and the field of rational functions invariant under the natural $G$-action is $\kk(V)^G$. The group algebra of $G$ over $\kk$ is $\kk G$.
\end{notation}

\begin{definition}\label{def:bfield-and-Dspan}
    Following \cite{blum2024degree}, the {\em field Noether number} $\bfield$ is
    \[
    \bfield:= \min(d\in \NN: \kk(V)^G\text{ is generated as a field by }\kk[V]^G_{\leq d}).
    \]
\end{definition}
\begin{definition}
    The {\em spanning degree} $\Dspan$ is
    \[
    \Dspan:= \min(d\in \NN: \kk[V]_{\leq d}\text{ spans }\kk(V)\text{ as }\kk(V)^G\text{-vector space}).
    \]
\end{definition}

\begin{notation}
    When we need to emphasize the dependence of $\bfield$ and $\Dspan$ on $G$ and $V$, we write $\bfield(G,V)$ and $\Dspan(G,V)$.
\end{notation}

The main result of this paper is Theorem~\ref{thm:main}. The proof is in Section~\ref{sec:main-result}. The inequality is sharp, as Section~\ref{ex:sharp} shows. Other results include:
\begin{itemize}
    \item A graded version of the Normal Basis Theorem  (Proposition~\ref{prop:normal-basis}), which serves as a lemma for Theorem~\ref{thm:main} but is proven in greater generality and may be independently useful.

    \item Several basic results on $\Dspan$, including:
    \begin{itemize}
        \item the inequalities and equalities mentioned above relating $\Dspan$ with $\Dreg$, $\Dirr$, and $\topdeg$ (Proposition~\ref{prop:c-Dreg-Dspan-topdeg});
        \item that $\Dspan(G,V)$ is monotonic nondecreasing in $G$ (Lemma~\ref{lem:monotonic-in-the-group}) and nonincreasing in (faithful) $V$ (Theorem~\ref{thm:monotonic-in-the-rep});
        \item that $\Dspan$ is subadditive along subdirect products (Proposition~\ref{prop:subadditivity});
        \item some general bounds on $\Dspan$: it never exceeds $|G|-1$ (Theorem~\ref{thm:refine-KP}), it is upper-bounded quadratically in the orbit lengths for permutation groups (Theorem~\ref{thm:permutations}), and it is bounded below by $\sqrt[n]{n!|G|}-O(n)$, where $n=\dim V$ (Proposition~\ref{prop:dim-counting-lower-bound} and Corollary~\ref{cor:explicit-lower-bound}).
    \end{itemize}
    \item An application of Theorem~\ref{thm:main} which, for $\kk=\CC$ and $G=C_p$ (the cyclic group of prime order $p$), bounds $\bfield(G,V)$ in terms of the character of $V$ (Proposition~\ref{prop:cyclotomic}).

\end{itemize}
Although Theorem~\ref{thm:main} requires the hypothesis $\Char \kk \nmid |G|$, the graded normal basis theorem and most of the results on $\Dspan$ do not. Indeed the general bound $\Dspan \leq |G|-1$ and the graded normal basis theorem are interesting primarily in the modular case. 

The monotonicity properties proven for $\Dspan$ both fail for $\bfield$.

\subsection{Methods and structure}\label{sec:methods}

The proof of the main result has the following steps:
\begin{enumerate}
    \item We consider the ideal $I$ in the ring $\KK[V]$, where $\KK:=\kk(V)^G$, that describes a generic orbit for the action of $G$ on $V$. The coefficients appearing in a generating set for this ideal form a generating set for the field $\kk(V)^G$ as an extension of $\kk$. Let $D_I$ be the minimal degree $d$ required for $\KK[V]_{\leq d}$ to contain such an ideal generating set.\label{step:generic-orbit-ideal}
    \item We show that $D_I\leq \Dspan + 1$.\label{step:bound-gen-orbit-degree}
    \item We view the ideal $I$ as describing linear relations over $\KK=\kk(V)^G$ between the various irreducible representations of $G$ occurring in a given isotypic component of $\kk[V]$, viewed as a $G$-representation. To find a generating set for $I$, it is sufficient to identify those linear relations occurring within $\KK[V]_{\leq D_I}$.\label{step:lin-rels-btw-irreducibles}
    \item We identify a graded copy of the regular representation $\Vreg$ inside $\kk[V]_{\leq \Dspan}$ that spans $\kk(V)$ over $\KK$. Using this, we write down matrix equations to find the linear relations over $\KK$ mentioned in Step~\ref{step:lin-rels-btw-irreducibles}, showing that the entries in the matrices can be taken to be invariant polynomials of degree at most $\max(\Dspan + D_I,2\Dspan)$. In view of Step~\ref{step:bound-gen-orbit-degree}, the result follows.\label{step:matrix-equations}
\end{enumerate}

The key lemma in Step~\ref{step:generic-orbit-ideal}, that the coefficients appearing in an ideal generating set for $I$ also generate $\KK$ as a field extension of $\kk$, closely follows fundamental work of M\"uller-Quade and Beth, Hubert and Kogan, and Kemper on algorithms to compute generators for a field of rational invariants \cite{muller1999calculating, hubert-kogan, kemper2007computation}. Step~\ref{step:bound-gen-orbit-degree} is a standard Gr\"obner basis argument. The main novelty of method in the present work is the point of view adopted in Step~\ref{step:lin-rels-btw-irreducibles}, of viewing the generic orbit ideal $I$ as determined by linear relations between copies of the same irreducible representation over the invariant field, which allows to write down the matrix equations in Step~\ref{step:matrix-equations} that prove the result.

While Theorem~\ref{thm:main} demands the non-modular hypothesis (i.e., that $|G|$ is not divisible by $\Char \kk$), we carry out as much as possible of the preparatory work over a field of any characteristic.

The paper is structured as follows. Section~\ref{sec:main-result} proves the main result. Section~\ref{sec:examples} concerns explicit examples, illustrating the proof technique in detail, and showing that the result is sharp and generalizes \cite{edidin2025orbit}. Section~\ref{sec:spanning-deg-bound} drops the hypothesis that $\Char \kk$ does not divide $|G|$ and gives the general results on $\Dspan$ mentioned in Section~\ref{sec:results} above. Section~\ref{sec:applications} gives a quick application of Theorem~\ref{thm:main}.

Here is a little more detail on Section 2. Steps~\ref{step:generic-orbit-ideal} and \ref{step:bound-gen-orbit-degree} in the above outline of Theorem~\ref{thm:main}'s proof are carried out in Section~\ref{sec:generic-orbit-ideal}, Steps~\ref{step:lin-rels-btw-irreducibles} and \ref{step:matrix-equations} are carried out in Section~\ref{sec:matrix-eqs}, and the pieces are assembled into the whole in Section~\ref{sec:assemble-proof}. Section~\ref{sec:graded-normal-basis} proves a graded version of the Normal Basis Theorem---the primary role of this is to provide the graded regular representation $\Vreg$ mentioned in Step~\ref{step:matrix-equations}, but we do some work to avoid the non-modular hypothesis because the result may be useful in other contexts. It is also shown in Section~\ref{sec:base-change} that we can assume $\kk$ is algebraically closed without loss of generality, which is used in Sections~\ref{sec:matrix-eqs}, \ref{sec:assemble-proof}, and part of Section~\ref{sec:graded-normal-basis} (and occasionally in later sections).  

\section{Proof of the main result}\label{sec:main-result}

In this section, we prove Theorem~\ref{thm:main}. Throughout, $G$ is a finite group, $\kk$ is a field, and $V$ is a finite-dimensional, faithful  representation of $G$ over $\kk$. 

The situation of Theorem~\ref{thm:main} is the non-modular one, where $\operatorname{char}\kk \nmid |G|$. However, we avoid imposing this assumption until Section~\ref{sec:matrix-eqs}, when it is used heavily.

For reference, the goal is to show, under these assumptions, that: 
\begin{center}
    {\em $\kk(V)^G$ is generated as a field by invariant polynomials of degree at most $2\Dspan + 1$.}
\end{center}

\begin{notation}\label{not:K}
    In what follows, we will be viewing $\kk(V)$ as a Galois field extension of $\kk(V)^G$, and as the image of a surjection of $\kk(V)^G$-algebras. So we adopt the notation 
    \[
    \KK := \kk(V)^G
    \]
    when we want to emphasize a view of the invariant field as a coefficient field or ground field.
\end{notation}

\subsection{Reduction to algebraically closed $\kk$}\label{sec:base-change}

Let $\FF$ be a field extension of $\kk$, and let $V_\FF = \FF\otimes_\kk V$ be the base change of $V$ to $\FF$, with the $G$-action given by the $G$-action on $V$ and trivial action on $\FF$. It is straightforward to check that $\Dspan$ is not affected by this base change:

\begin{lemma}\label{lem:base-change}
    We have
    \[
    \Dspan(G,V_\FF) = \Dspan(G,V).
    \]
\end{lemma}

\begin{proof}
    View $\kk(V)$ as a subfield of $\FF(V_\FF)$ via the inclusion extended from $\kk\hookrightarrow \FF$. 
    
    There is a graded, $G$-equivariant isomorphism of $\FF[V_\FF]$ with $\FF\otimes_\kk \kk[V]$ (where $\FF$ is viewed as a graded $\kk$-algebra concentrated in degree $0$). Thus in particular, (i) $\FF[V_\FF]_{\leq d} \cong \FF\otimes_\kk \kk[V]_{\leq d}$ for any $d$, and (ii) $\FF[V_\FF]^G \cong  (\FF\otimes_\kk \kk[V])^G \cong \FF\otimes_\kk \kk[V]^G$, where the first isomorphism is because of the equivariance of the isomorphism $\FF[V_\FF]\cong \FF\otimes_\kk \kk[V]$, and the second is because the functor of invariants commutes with flat base change.
    
    From $\FF[V_\FF]_{\leq d} \cong \FF\otimes_\kk \kk[V]_{\leq d}$ it follows that for any $d$, $\FF[V_\FF]_{\leq d}$ spans $\FF(V_\FF)$ over $\FF(V_\FF)^G$ if and only if $\kk[V]_{\leq d}\subseteq \FF(V_\FF)$ does so. We thus have to show that $\kk[V]_{\leq d}$ spans $\FF(V_\FF)$ over $\FF(V_\FF)^G$ if and only if it spans $\kk(V)$ over $\kk(V)^G$. By dimension counting, these conditions hold if and only if $\kk[V]_{\leq d}$ contains $|G|$ elements linearly independent, respectively, over $\FF(V_\FF)^G$ and over $\kk(V)^G$.

    The ability to clear denominators means that, for any set $S$ of elements of $\kk[V]_{\leq d}$, $S$ is linearly independent over $\kk(V)^G$ if and only if it is linearly independent over $\kk[V]^G$, and linearly independent over $\FF(V_\FF)^G$ if and only if it is linearly independent over $\FF[V_\FF]^G$.

    But since $\FF[V_\FF]^G \cong \FF\otimes_\kk \kk[V]^G$, linear independence of $S\subseteq \kk[V]_{\leq d}$ over $\FF[V_\FF]^G$ is equivalent to linear independence over $\kk[V]^G$.  This completes the proof.
\end{proof}

It is also the case that $\bfield(G,V_\FF) = \bfield(G,V)$, by \cite[Lemma~2.1]{blum2024degree}. Thus, in proving Theorem~\ref{thm:main}, it costs us no generality to assume that $\kk$ is algebraically closed. We make this assumption in Sections~\ref{sec:matrix-eqs} and \ref{sec:assemble-proof} below, and part of Section~\ref{sec:graded-normal-basis}.

\subsection{The generic orbit ideal}\label{sec:generic-orbit-ideal}

Viewing both $\kk(V)^G$ and $\kk[V]$ as subrings of $\kk(V)$, consider the natural multiplication homomorphism

\begin{equation}\label{eq:Xi}
\Xi: \kk(V)^G\otimes_\kk \kk[V] \rightarrow \kk(V).
\end{equation}

This is evidently a $G$-equivariant map. Note that because $\kk(V)$ is finite (of degree $|G|$) over $\kk(V)^G$, the coordinate functions on $V$ are integral over $\kk(V)^G$, so the map $\Xi$ is surjective. 

\begin{definition}\label{def:generic-orbit-ideal}
Define the {\em generic orbit ideal} as 
\[
I:= \ker \Xi.
\]
\end{definition}

The reason for the name is as follows. Suppose $\phi:X\rightarrow Y$ is a dominant morphism of reduced, irreducible affine $\kk$-schemes. Then $\phi^*:\kk[Y]\rightarrow \kk[X]$ is an injective homomorphism of integral domains, and we may identify $\kk[Y]$ with its image $\phi^*(\kk[Y])$ in $\kk[X]$. The {\em generic fiber of $\phi$}, i.e., the fiber $\Spec \kk(Y)\times_Y X$ over the generic point of $Y$, embeds canonically as a closed subscheme in $X_{\kk(Y)}$, the base change of $X$ to the field of rational functions on $Y$; the embedding is dual to the canonical map
\[
\Xi:\kk(Y)\otimes_\kk \kk[X]\rightarrow \kk(Y)\otimes_{\kk[Y]} \kk[X]
\]
on the coordinate rings. Thus, the generic fiber of $\phi:X\rightarrow Y$ is isomorphically the closed subscheme of $X_{\kk(Y)}$ defined by the ideal $\ker \Xi$. Furthermore, $\kk(Y)\otimes_{\kk[Y]}\kk[X]$, as a localization of $\kk[X]$, can be identified with a subring of $\kk(X)$, the field of rational functions on $X$, and composition with this inclusion turns $\Xi$ into the natural multiplication map
\[
\kk(Y)\otimes_\kk \kk[X]\rightarrow \kk(X).
\]
If $\phi:X\rightarrow Y$ is a {\em finite} map, then $\kk(Y)\otimes_{\kk[Y]}\kk[X]$ is all of $\kk(X)$ (e.g., by \cite[Proposition~16]{kohls2014top}), so $\Xi$ is surjective onto $\kk(X)$, and the generic fiber of $\phi$ is isomorphically $\Spec \kk(X)$.

Definition~\ref{def:generic-orbit-ideal} comes from applying all this with $X=V$ and $Y=V\sslash G=\Spec \kk[V]^G$, the categorical quotient of $V$ by $G$. Because $G$ is finite, $V\sslash G$ is in fact a {\em geometric} quotient, i.e., the points of $V\sslash G$ are in bijection with the orbits of $G$ on $V$. Then the generic fiber of the finite map $V\rightarrow V\sslash G$ may be called the {\em generic orbit} of $G$ on $V$. Define 
\[
\KK:=\kk(V)^G=\Frac(\kk[V]^G) = \kk(V\sslash G),
\]
the field of rational functions on the quotient $V\sslash G$. (The equality $\kk(V)^G=\Frac(\kk[V]^G)$ is guaranteed by the finiteness of $G$.) The above considerations show that the ideal $I$ of Definition~\ref{def:generic-orbit-ideal} defines the generic orbit of the action of $G$ on $V$ as a closed subscheme of the affine space $V_{\KK}$, hence the name {\em generic orbit ideal}.

Although both tensor factors of the domain of $\Xi$ are subrings of the field $\kk(V)$, it is clarifying for what follows to distinguish between them notationally. Suppose $n = \dim_\kk V$; let $x_1,\dots,x_n$ be a basis of coordinate functions on $V$, so $\kk(V) = \kk(x_1,\dots,x_n)$. Then
\[
\kk(V)^G\otimes_\kk \kk[V] \cong \KK[X_1,\dots,X_n]
\]
for new indeterminates $X_1,\dots,X_n$ subject to the same $G$-action (defined over $\kk$) as $x_1,\dots,x_n$. We think of these new indeterminates as a basis of coordinate functions on the $V$ that appears in the second tensor factor in \eqref{eq:Xi}, whereas the $V$'s appearing in the first tensor factor and in the codomain both have $x_1,\dots,x_n$ as coordinate functions. Henceforward, we identify $\kk(V)^G\otimes_\kk \kk[V]$ with $\KK[X_1,\dots,X_n]$; then the map $\Xi$ can be viewed as the map
\begin{equation}\label{eq:Xi-X}
\KK[X_1,\dots,X_n]\rightarrow \kk(V)
\end{equation}
obtained by canonically embedding $\KK=\kk(V)^G$ in $\kk(V)$, and meanwhile mapping $X_j \mapsto x_j$ for  $j=1,\dots, n$.

For our purposes, the importance of the generic orbit ideal is that it is related to generation of $\KK$, by the following key lemma.

\begin{lemma}\label{lem:coeffs-generate}
    Suppose $I$ contains some polynomials $f_1,\dots,f_m\in \KK[X_1,\dots,X_n]$. If $f_1,\dots,f_m$ generate the ideal $I$, then the set of all coefficients of all the $f_j$'s generates $\KK=\kk(V)^G$ as a field extension of $\kk$.
\end{lemma}
\begin{remark}
Results similar to this lemma appear in~\cite{muller1999calculating,hubert-kogan,kemper2007computation}, where they are used for algorithms to compute generators of the invariant field. M\"uller-Quade and Beth gave an algorithm for generators of the invariant field
that uses the ideal $I$  to compute generators of the field $\KK=\kk(V)^G$ over $\kk$ 
for a linear action of an algebraic group $G$. The main ideas are as follows. Because a priori generators of $\kk(V)^G$ are unknown, we cannot compute in $\kk(V)^G\otimes_\kk \kk[V]$. However, we can view $\kk(V)^G\otimes_\kk \kk[V]\cong \KK[X_1,X_2,\dots,X_n]$ as a subring of $\kk(V)\otimes_\kk \kk[V]=\kk(x_1,x_2,\dots,x_n)[X_1,X_2,\dots,X_n]$.
If $\widetilde{I}$ is the ideal in $\kk(x_1,x_2,\dots,x_n)[X_1,X_2,\dots,X_n]$
generated by $I$ and one computes a reduced Gr\"obner basis of  $\widetilde{I}$ then the coefficients appearing in the Gr\"obner basis will automatically lie in the subfield $\KK$, and by the lemma above they will generate it. Rather than computing in $\kk(x_1,x_2,\dots,x_n)[X_1,X_2,\dots,X_n]$ which would be very inefficient, M\"uller-Quade and Beth compute in the polynomial ring $\kk[x_1,x_2,\dots,x_n,X_1,X_2,\dots,X_n]$ using some additional results. Hubert and Kogan generalized these results to nonlinear polynomial or even rational actions, and they showed in~\cite{hubert-kogan} that one also can use similar ideas to  express arbitrary invariant rational functions in terms of the generators. Kemper generalized this further to algebraic group actions on irreducible varieties~\cite{kemper2007computation}.
\end{remark}

\begin{proof}
    Assume $(f_1,\dots,f_m)=I$. Computing a Gr\"obner basis for $I$ from the $f_j$'s involves only the field operations starting from the coefficients of the $f_j$'s, so if the coefficients of such a Gr\"obner basis generate $\KK$ over $\kk$, then the coefficients of the $f_j$'s already do so. Thus we can assume without loss of generality that $f_1,\dots,f_m$ form a Gr\"obner basis for $I$ (with respect to some term order). 
    
    Let $\KK'$ be the field extension of $\kk$ generated by the coefficients of the $f_j$'s; clearly $\KK'\subseteq \KK$, and our goal is to show that $\KK'=\KK$. It is enough to show that $\KK'$ contains $\kk[V]^G\subseteq \kk(V)^G$, because $\KK = \kk(V)^G= \operatorname{Frac}\kk[V]^G$. 
    
    Now take any element $f\in\kk[V]^G$, replace the $x_i$'s with $X_i$'s, and call it $F\in \KK[X_1,\dots,X_n]$. Note that, by construction, all the coefficients of $F$ actually belong to $\kk$. Then apply the division algorithm to $F$ with respect to the Gr\"obner basis $(f_1,\dots,f_m)$ to compute the unique representative of $F+I$ in $\KK[X_1,\dots,X_n]$ that is reduced with respect to this Gr\"obner basis. 
    
    This representative is $f\in\KK$, i.e., $f$ viewed as a constant polynomial in $\KK[X_1,\dots,X_n]$; we see this as follows. Note that $f\in\KK$ can be viewed either as an element of $\KK[X_1,\dots,X_n]$, i.e., the domain of $\Xi$, or as an element of $\kk(V)$, the codomain; and we have $\Xi(F)=f=\Xi(f)$. Therefore, $F-f\in \ker \Xi = I$, so $f$ represents $F$'s class in $\KK[X_1,\dots,X_n]/I$. Furthermore, as $f$, viewed as an element of $\KK[X_1,\dots, X_n]$, has no nonconstant term, it is reduced with respect to the Gr\"obner basis $(f_1,\dots,f_m)$.
    
    So we can conclude that, applying the division algorithm to $F$, we end up with $f$. But the division algorithm only uses field operations on the coefficients of $F$ and the coefficients appearing in $f_1,\dots,f_m$, so it must be that $f$ is rationally expressible in terms of these coefficients. But all the coefficients of $F$ are in $\kk\subseteq\KK'$, and all the coefficients of $f_1,\dots,f_m$ are in $\KK'$ by construction. Thus, $f\in\KK'$. Since $f\in\kk[V]^G$ was arbitrary, this proves that $\KK'\supset \kk[V]^G$, and thus $\KK'=\KK$.
\end{proof}

\begin{definition}
    Define the {\em generic orbit degree} as
    \[
    D_I := \min(d\in\NN : I\text{ is generated by polynomials of degree}\leq d).
    \]
\end{definition}

By a standard Gr\"obner basis argument, $D_I$ is bounded in terms of $\Dspan$:

\begin{lemma}\label{lem:DI-and-Dspan}
    We have
    \[
    D_I \leq \Dspan + 1.
    \]
\end{lemma}

\begin{proof}
    Choose a term order that respects degree (for example, graded lexicographic or graded reverse lexicographic order), and let $\mcB := \{f_1,\dots,f_m\}$ be a reduced Gr\"obner basis for $I$ with respect to this term order. By the definition of $D_I$ and the fact that $f_1,\dots,f_m$ generate $I$, we have
    \[
    D_I  \leq \max_j \deg f_j.
    \]
    By Gr\"obner basis theory, the monomials of $\KK[X_1,\dots,X_n]$ in the complement of the initial ideal $\init(I)$ of $I$ represent $I$-residue classes in $\KK[X_1,\dots,X_n]/I\cong \kk(V)$ that  form a basis over $\KK$. Let $\mcM$ be the set of these monomials; it is a finite set because $\kk(V)$ is finite over $\KK$. Let
    \[
    d:=\max(\deg m: m\in \mcM)
    \]
    be the maximum degree of an element of $\mcM$, and let $m^\star\in\mcM$ be some element that achieves this degree. 
    It is clear from the definition of $\Dspan$ that $\Dspan \leq d$, 
    and in fact we have 
    \[
    \Dspan = d,
    \]
    as follows. For any $f$ in the $\KK$-span $\langle \mcM \rangle_\KK$ of $\mcM$, $f$ is the unique representative of its class $f+I$ that is reduced with respect to $\mcB$. Because the chosen term order respects degree, this also means that $f$ is of minimal degree inside $f+I$. In particular, every element of $m^\star + I$ has degree at least $d$. Thus, the $\KK$-vector space $\KK[X_1,\dots,X_n]_{\leq d-1}$ does not contain any representative of $m^\star + I$, and it follows that its image in $\KK[X_1,\dots,X_n]/I \cong \kk(V)$ is a proper $\KK$-subspace. Therefore, $\Dspan > d-1$, and thus $\Dspan=d$.

    Now by the definition of $d$, every monomial of degree at least $d+1$ is contained in $\init(I)$. It follows that $\init(I)$ does not have any minimal generators of degree greater than $d+1$. Since $\mcB$ is reduced, every leading term of every $f_j$ is a minimal generator of $\init(I)$; because the term order respects degree, it follows that no $f_j$ has degree greater than $d+1$. Putting everything together, we have
    \[
    D_I\leq \max_j \deg f_j \leq d+1 = \Dspan + 1,
    \]
    which gives the desired inequality.
\end{proof}

\subsection{A graded normal basis theorem}\label{sec:graded-normal-basis}

The Normal Basis Theorem from Galois theory (e.g., \cite[Theorem~28]{artin-galois}) asserts that $\kk(V)$ has a $\KK=\kk(V)^G$-basis consisting of a single $G$-orbit (i.e., a {\em normal basis}). In other words, as a representation of $G$ over $\KK=\kk(V)^G$, $\kk(V)$ is nothing but the regular representation. In this section, we prove the mild refinement in the present context that the normal basis can be taken to span a graded subspace of $\kk[V]_{\leq \Dspan}$, which we will call $\Vreg$. The statement is Proposition~\ref{prop:normal-basis} below.

We use $\Vreg$ in an essential way in Section~\ref{sec:matrix-eqs} toward the proof of our main result. That section requires the non-modular assumption $\Char \kk\nmid |G|$, in which case, the existence of $\Vreg$ is  straightforward. However, we avoid that assumption in this section in order to give Proposition~\ref{prop:normal-basis} wider applicability; thus the proof is more involved. 

Our proof of  Proposition~\ref{prop:normal-basis} does demand that $\kk$ be algebraically closed (because Lemma~\ref{lem:indecomposable} fails without this assumption). However, ideas discussed here (in particular the Krull-Remak-Schmidt and Noether-Deuring theorems) will also be used below in Section~\ref{sec:spanning-deg-bound} to yield that $\Dreg\leq \Dspan$ without the assumption of algebraic closure. We do not know whether Proposition~\ref{prop:normal-basis} itself holds without this assumption.

As we will work with group algebras, we will need several results about finite dimensional algebras and their finite dimensional modules. In particular we will use some results regarding how algebras and modules behave under base field extension.
In our discussion below we will assume that 
$R$ is a finite dimensional $\kk$-algebra, and all modules are finite dimensional over $\kk$. The main use case will be $R=\kk G$.

For $R$-modules $A$ and $B$ we write $\Hom_R(A,B)$ for the set of $R$-module homomorphisms, $\End_R(A)$ for $\Hom_R(A,A)$, and $1_A\in\End_R(A)$ for the identity map on $A$. By the Krull-Remak-Schmidt Theorem, every module is the direct sum of indecomposable modules, and the isomorphism types of the summands are unique up to permutation. A module $A$ is indecomposable if and only if $\End_R(A)$ is a local ring. Thus, if $A$ is indecomposable then $\End_R(A)$ has a unique maximal ideal ${\mathfrak m}$ such that all elements not in ${\mathfrak m}$ are invertible. The ideal ${\mathfrak m}$ is the Jacobson radical, and consists exactly of all nilpotent elements in $\End_R(A)$.

We first check that a splitting of a surjection onto a projective module guaranteed to exist by projectivity, can be taken to be compatible with any given decomposition of the domain into indecomposables.

\begin{lemma}\label{lem:section}
Suppose that $A$ is a projective $R$-module, $B_1,B_2,\dots,B_\ell$ are indecomposable  $R$-modules, and $\phi=[\phi_1\ \phi_2\ \dots\ \phi_\ell]:\bigoplus_{i=1}^\ell B_i\to A$ is a surjective $R$-module homomorphism. 
Then there exists a  subset ${\mathcal S}\subseteq \{1,2,\dots,\ell\}$ such that the restriction of $\phi$ to $\bigoplus_{i\in {\mathcal S}} B_i$ is an isomorphism onto $A$.
\end{lemma}
\begin{proof}
We prove the statement by induction on $\ell$. Let $B=\bigoplus_{i=1}^\ell B_i$. The case $\ell=0$ holds vacuously.
Because $A$ is projective, $\phi$ has a section 
$$\psi=\begin{bmatrix}\psi_1\\ \psi_2\\ \vdots\\ \psi_\ell\end{bmatrix}:A\to B,
$$
so that $\phi\psi=\sum_{i=1}^\ell \phi_i\psi_i=1_A$. 

First assume  that $A$ is indecomposable. Now $\End_R(A)$ is a local ring with a unique maximal ideal ${\mathfrak m}$. Since $1_A\not\in {\mathfrak m}$, $\phi_i\psi_i\not\in {\mathfrak m}$ for some $i$. This means that $\phi_i\psi_i$ is an automorphism of $A$. It follows that $\psi_i(A)$ is a direct summand of $B_i$. Because $B_i$ is indecomposable, $\psi_i(A)=B_i$ and $\phi_i:B_i\to A$ is the desired isomorphism.

Now suppose that $A$ is not indecomposable. By the Krull-Remak-Schmidt Theorem, we can write $A=A_1\oplus A_2$ where $A_1$ is indecomposable. Let $\pi_j:A\to A_j$ be the projection for $j=1,2$. From the indecoposable case, we know there exists an $i$ such that $\pi_1\phi_i:B_i\to A_1$ is an isomorphism. By permuting the summands of $B$, we can assume without loss of generality that $\pi_1\phi_1:B_1\rightarrow A_1$ is an isomorphism. It follows that $\phi_1$ is an isomorphism to its image in $A$. Then $\phi_1(\pi_1\phi_1)^{-1}\pi_1\in \End_R(A)$ is an idempotent projection to $\phi_1(B_1)$, and $1_A-\phi_1(\pi_1\phi_1)^{-1}\pi_1$ projects to $\phi_1(\pi_1\phi_1)^{-1}\pi_1$'s kernel, which is an $R$-module complement for $\phi_1(B_1)$. Thus, replacing $A_1$ with $\phi_1(B_1)$, $\pi_1$ with $\phi_1(\pi_1\phi_1)^{-1}\pi_1$, $A_2$ with $\ker \phi_1(\pi_1\phi_1)^{-1}\pi_1$, and $\pi_2$ with $1_A - \phi_1(\pi_1\phi_1)^{-1}\pi_1$, we (still have $\pi_1\phi_1$ an isomorphism, and) may assume without loss of generality that $\pi_2\phi_1=0$.
We write 
\[
B_{\geq 2}=\bigoplus_{i=2}^\ell B_i
\]
and $\phi_{\geq 2}:B_{\geq 2}\to A$  given by 
\[
\phi_{\geq 2} :=
\begin{bmatrix}
    \phi_2 & \dots & \phi_\ell
\end{bmatrix}.
\]
Because
$$
\phi=\begin{bmatrix} \pi_1\\ \pi_2\end{bmatrix}\begin{bmatrix}
    \phi_1 &\phi_{\geq 2} 
    \end{bmatrix}=\begin{bmatrix}
        \pi_1\phi_1 & \pi_1\phi_{\geq 2}\\
        0 & \pi_2\phi_{\geq 2} 
    \end{bmatrix}
$$
is surjective, we get that $\pi_2\phi_{\geq 2}:B_{\geq 2}\to A_2$ is surjective.
The module $A_2$ is a direct summand of the projective module $A$, so it is also projective. By induction, there exists a subset $S'\subseteq \{2,3,\dots,\ell\}$ such that the restriction of $\pi_2\phi_{\geq 2}$ to $\bigoplus_{i\in S'} B_i\to A_2$ is an isomorphism. If we take $S=S'\cup \{1\}$, then the restriction of $\phi$ to $\bigoplus_{i\in S} B_i$ is an isomorphism.
\end{proof}

Let $\FF/\kk$ be an arbitrary field extension. We write $\widetilde{R}=\FF\otimes_\kk R$, and, for an $R$-module $A$, we write $\widetilde{A}=\FF\otimes_\kk A$. Note that
\begin{equation}\label{eq:End-and-base-change}
\End_{\widetilde{R}}(\widetilde{A}) \cong \Hom_R(A,\widetilde{A}) = \Hom_R(A,\FF\otimes_\kk A)\cong \FF\otimes_\kk\End_R(A),
\end{equation}
with the first isomorphism by the restriction-extension adjunction, and the last because $A$ is a finitely presented $R$-module and $\FF$ is flat over $\kk$.

We will need the following classical result which can for example be found in \cite[(29.7), (29.11)]{curtis1966representation}:

\begin{theorem}[Noether-Deuring Theorem]\label{theo:NoetherDeuring}
If $\widetilde{A}$ and $\widetilde{B}$ are isomorphic as $\widetilde{R}$-modules, then $A$ and $B$ are isomorphic $R$-modules.
 \end{theorem}

We also need the fact that base changing from an algebraically closed field preserves indecomposability:

\begin{lemma}\label{lem:indecomposable}
Suppose that $\kk$ is algebraically closed.
   If $A$ is an indecomposable $R$-module, then $\widetilde{A}$ is an indecomposable $\widetilde{R}$-module.
\end{lemma}
\begin{proof}
Since $A$ is indecomposable, the algebra $T=\End_R(A)$ is a local ring with a unique maximal ideal ${\mathfrak m}$. The quotient $T/{\mathfrak m}$ is a skew-field and a finite dimensional $\kk$-algebra. Because $\kk$ is algebraically closed, 
$T/{\mathfrak m}\cong \kk$.
If we apply base field extension from $\kk$ to $\FF$ for the exact sequence
$$
\xymatrix{
0\ar[r] & {\mathfrak m}\ar[r] &  T \ar[r] & \kk\ar[r] & 0}
$$
we get an exact sequence
$$
\xymatrix{
0\ar[r] & \widetilde{\mathfrak m}\ar[r] &  \widetilde{T} \ar[r] & \FF\ar[r] & 0}
$$
of $\widetilde{T}$-modules. Since the quotient $\widetilde{T}/\widetilde{\mathfrak m}$ is isomorphic to the field $\FF$, every element of $\widetilde{T}$ has the form $\alpha 1_{\widetilde A} + m$ for $\alpha\in \FF$ and $m\in \widetilde{\mathfrak m}$. Because ${\mathfrak m}$ is a nilpotent ideal, $\widetilde{\mathfrak m}$ is also nilpotent, and in particular, $m$ is nilpotent. So $\alpha 1_{\widetilde A} + m$ is invertible provided $\alpha\neq 0$. Thus $\widetilde{T} \setminus \widetilde{\mathfrak m}$ consists of units. This proves that $\widetilde{\mathfrak m}$ is the unique maximal ideal of $\widetilde{T}$ and that $\widetilde{T}$ is a local ring. Because $\widetilde{T}\cong \End_{\widetilde{R}}(\widetilde{A})$ by \eqref{eq:End-and-base-change}, it follows that $\widetilde{A}$ is an indecomposable $\widetilde{R}$-module. 
 \end{proof}

We are ready to prove this section's main result.

\begin{proposition}[Graded normal basis theorem]\label{prop:normal-basis}
Suppose that $\kk$ is algebraically closed. 
    There exists a $\kk G$-module direct summand $\Vreg$ of $\kk[V]_{\leq \Dspan}$ that is isomorphic to the regular representation of $G$ over $\kk$, is graded as a $\kk$-subspace of $\kk[V]_{\leq \Dspan}$, and such that any $\kk$-basis for it is also a $\kk(V)^G$-basis for $\kk(V)$.
\end{proposition}

\begin{proof}
Let $R=\kk G$ and $\FF=\KK$, so that $\widetilde{R}=\KK G$.
We can write $\kk[V]_{\leq \Dspan}$ as a direct sum of indecomposable $R$-modules, with each summand contained in a homogeneous component: 
$$
\kk[V]_{\leq \Dspan}= B_1\oplus B_2\oplus \cdots \oplus B_s.
$$
 Restriction of $\Xi$ gives a surjective $\widetilde{R}$-module homomorphism
$$\bigoplus_{i=1}^r \widetilde{B}_i=\KK\otimes_\kk \kk[V]_{\leq \Dspan}\rightarrow \kk(V)\cong \widetilde{R}.
$$
Since $\widetilde{R}$ is a projective  $\widetilde{R}$-module, 
and $\widetilde{B}_i$ is an indecomposable $\widetilde{R}$-module for all $i$ by Lemma~\ref{lem:indecomposable},  we can apply  Lemma~\ref{lem:section}. So there exists a subset $S\subseteq \{1,2,\dots,s\}$ such that the restriction of $\Xi$ to $\bigoplus_{i\in S}\widetilde{B}_i$ is a $\widetilde{R}$-module isomorphism. 
We take $V_{\rm reg}=\bigoplus_{i\in S}B_i$. 
Then $\widetilde{V}_{\rm reg}=\KK\otimes_{\kk}V_{\rm reg}\cong \widetilde{R}\cong \kk(V)$. This implies that any $\kk$-basis of $V_{\rm reg}$ will be a $\KK$-basis of $\KK G\cong \kk(V)$.
Also, by Theorem~\ref{theo:NoetherDeuring}, $\widetilde{V}_{\rm reg}\cong \widetilde{R}$ implies that $V_{\rm reg}$ is isomorphic to $R$ as an $R$-module.
\end{proof}

\subsection{Matrix equations for coefficients}\label{sec:matrix-eqs}

This subsection is the heart of the argument. We split the low-degree components of $\kk[X_1,\dots,X_n]$ into irreducible subrepresentations of $G$ and write down $\kk$-bases for these subrepresentations. We then view the elements of $I$ of low degree as linear combinations of these basis elements with coefficients in $\KK = \kk(V)^G$, and use linear algebra to find the coefficients appearing in these linear combinations in terms of invariants of low degree. It is here (and in the following section) that we instate, and make heavy use of, the assumption that  $\operatorname{char}\kk$ does not divide the order of $G$. 

We also assume that $\kk$ is algebraically closed. Per Lemma~\ref{lem:base-change}, this costs no generality.

Let $\Irr_\kk(G)$ be an index set for the isomorphism classes of irreducible representations of $G$ over $\kk$. For each $\lambda\in \Irr_\kk(G)$, fix a particular representation $V_\lambda$, of dimension $d_\lambda$, and a specific $\kk$-basis
\begin{equation}\label{eq:basis-for-Vlambda}
v_1^\lambda,\dots,v_{d_\lambda}^\lambda \in V_\lambda.
\end{equation}
Now we break the graded components of the polynomial ring $\kk[V]$ into irreducible subrepresentations of $G$. We know by Proposition~\ref{prop:normal-basis} that we can find enough such irreducible subrepresentations in degree $\leq \Dspan$ to constitute a copy $\Vreg$ of the regular representation, with the property that a $\kk$-basis for it is also a $\KK$-basis for $\kk(V)$; we have
\begin{equation}\label{eq:base-change-reg}
\KK\otimes_\kk \Vreg \cong \kk(V)
\end{equation}
as $G$-representations over $\KK$, with the isomorphism given by restricting the multiplication map $\Xi$.

Because we have assumed that $\kk$ is algebraically closed, the multiplicity of each $V_\lambda$ in $\Vreg$ is also the degree $d_\lambda:=\dim_\kk V_\lambda$. For each $\lambda\in\Irr_\kk(G)$, we fix a $\kk$-basis of $G$-module embeddings of $V_\lambda$ in $\Vreg$:
\begin{equation}\label{eq:psis}
    \psi_1^\lambda,\dots,\psi_{d_\lambda}^\lambda\in\Hom_{\kk G}(V_\lambda,\Vreg).
\end{equation}
We can take each $\psi^\lambda_j$ to have image contained in a single graded component of $\Vreg$.

\begin{observation}\label{obs:K-basis}
    The $\psi_1^\lambda,\dots,\psi_{d_\lambda}^\lambda$ also form a $\KK=\kk(V)^G$-basis for
    \[
    \Hom_{\kk G}(V_\lambda, \kk(V))
    \]
    with respect to the $\KK$-action on $\Hom_{\kk G}(V_\lambda,\kk(V))$ induced from multiplication in $\kk(V)$.
\end{observation}

\begin{proof}
    We have 
    \[
    \Hom_{\kk G}(V_\lambda, \kk(V))\cong \Hom_{\kk G}(V_\lambda, \KK \otimes_\kk \Vreg)\cong \KK \otimes_\kk \Hom_{\kk G}(V_\lambda, \Vreg),
    \]
    with the first isomorphism by \eqref{eq:base-change-reg}, and the second because $V_\lambda$ is finitely presented projective as a $\kk G$-module. Since the isomorphism in \eqref{eq:base-change-reg} is the multiplication map, the action of $\KK$ on $\Hom_{\kk G}(V_\lambda,\kk(V))$ is the one induced from multiplication in $\kk(V)$. The statement follows.
\end{proof}

It is then immediate that the $|G| = \sum_{\lambda\in \Irr_\kk(G)} d_\lambda^2$ elements
\begin{equation}\label{eq:basis-for-Vreg}
\psi_i^\lambda\left(v_j^\lambda\right)
\end{equation}
of $\Vreg\subseteq \kk[V]$, for $\lambda\in \Irr_\kk(G)$ and $1\leq i,j\leq d_\lambda$, form a $\kk$-basis of $\Vreg$ and thus also a $\KK$-basis of $\kk(V)$. All of these elements of $\kk[V]$ are homogeneous of degree at most $\Dspan$.

We now play the same game for $\kk[V]_{\leq D_I}$, decomposing each of its graded components into irreducible $G$-representations, yielding a direct-sum decomposition
\[
\kk[V]_{\leq D_I} = \bigoplus_{s\in S} W_s,
\]
where $S$ is a finite set indexing this decomposition. For each $s\in S$, there is a unique $\lambda\in \Irr_\kk(G)$, which we denote $\lambda(s)$, such that $W_s\cong V_\lambda$ as $\kk G$-modules, and we fix a $\kk G$-isomorphism
\[
\phi_s:V_{\lambda(s)} \rightarrow W_s
\]
accordingly. Then the elements
\begin{equation}\label{eq:basis-for-leqDI}
\left\{\phi_s\left(v^{\lambda(s)}_j\right)\right\}_{s\in S,\; 1\leq j \leq d_{\lambda(s)}}
\end{equation}
form a $\kk$-basis of $\kk[V]_{\leq D_I}$, and so also a $\KK$-basis of $\KK\otimes_\kk \kk[V]_{\leq D_I}= \KK[X_1,\dots,X_n]_{\leq D_I}$. Each of them is a homogeneous polynomial of degree at most $D_I$, which we view as a polynomial in the $x_i$ when working inside $\kk(V)$, and in the $X_i$ when working inside $\KK[X_1,\dots,X_n]$.

Our plan (carried out below and in Section~\ref{sec:assemble-proof}) is to write a generating set for $I$ in terms of this basis. By definition of $D_I$, $I$ is generated by its elements of degree $\leq D_I$; this is the target generating set. Viewing the elements \eqref{eq:basis-for-leqDI} as belonging to the second tensor factor in the domain of \eqref{eq:Xi}, an element of $I$ of degree $\leq D_I$ can be written as a $\KK=\kk(V)^G$-linear combination of the elements \eqref{eq:basis-for-leqDI}. Explicitly, in the notation of \eqref{eq:Xi-X}, we view the elements \eqref{eq:basis-for-leqDI} as homogeneous polynomials over $\kk$ in the $X_i$, and because they form a $\KK$-basis for $\KK[X_1,\dots,X_n]_{\leq D_I}$, any element of $I$ of degree $\leq D_I$ can be written as a $\KK$-linear combination of them which becomes zero after substituting $X_i\mapsto x_i$. We aim to show that the coefficients that appear in such a linear combination can be calculated, via the field operations, from elements of $\kk[V]^G$ whose degrees are controlled. Via Lemma~\ref{lem:coeffs-generate}, this will yield the desired result.

The strategy for writing these coefficients in terms of elements of $\kk[V]^G$ of controlled degree, is to represent the elements \eqref{eq:basis-for-leqDI} in terms of the $\KK$-basis \eqref{eq:basis-for-Vreg} for $\kk(V)$. The $\KK$-coordinates of each element \eqref{eq:basis-for-leqDI} in terms of the basis \eqref{eq:basis-for-Vreg} will be calculated via a matrix equation where the entries in the matrices have degree controlled in terms of $D_I$ and $\Dspan$. Then, the $\KK$-linear combinations of the elements \eqref{eq:basis-for-leqDI} which become trivial after the substitution $X_i\mapsto x_i$ (i.e., the elements of $I$ of low degree) can be calculated using elementary linear algebra from these coordinates. The details follow. The first step is to make sure that some matrices that appear in the matrix equations are nonsingular.

Recall the {\em Reynolds operator}
\begin{align*}
    \Reynolds: \kk[V]&\rightarrow \kk[V]^G\\
     f&\mapsto \frac{1}{|G|}\sum_{g\in G} gf.
\end{align*}
Note that $\Reynolds$ preserves degree, and also extends naturally to $\kk(V)\rightarrow \kk(V)^G$, so that 
\[
\Reynolds(f) = |G|^{-1}\Tr_{\kk(V)/\KK}(f),
\]
where $\Tr$ is the trace in the sense of field theory.

\begin{lemma}\label{lem:nonsingular}
    Fix $\mu\in \Irr_\kk(G)$, and $k$ with $1\leq k \leq d_\mu$ (indexing a specific $v^\mu_k\in V_\mu$ as in \eqref{eq:basis-for-Vlambda}). Then there exist $d_\mu$ homogeneous polynomials $h_1,\dots, h_{d_\mu} \in \kk[V]$, of degree at most $\Dspan$, such that the $d_\mu \times d_\mu$ matrix
    \[
    \begin{pmatrix}
        \Reynolds\left(h_1\psi^\mu_1\left(v^\mu_k\right)\right)& \dots & \Reynolds\left(h_1\psi^\mu_{d_\mu}\left(v^\mu_k\right)\right) \\
        \vdots & \ddots & \vdots \\
        \Reynolds\left(h_{d_\mu}\psi^\mu_1\left(v^\mu_k\right)\right)& \dots & \Reynolds\left(h_{d_\mu}\psi^\mu_{d_\mu}\left(v^\mu_k\right)\right)
    \end{pmatrix},
    \]
    with entries in $\kk[V]^G\subseteq \KK$, is nonsingular.
\end{lemma}

\begin{remark}
    The proof shows that $h_1,\dots, h_{d_\mu}$ can actually be chosen from among the $\psi^{\nu}_j\left(v^{\nu}_\ell\right)$, $\nu \in \Irr_\kk(G)$, $1\leq j,\ell \leq d_{\nu}$, and a close read (with Schur's lemma in mind) shows that $\nu$ must be dual to $\mu$.
\end{remark}

\begin{proof}
    The field extension $\kk(V) / \KK$, being Galois (with group $G$), is separable. Therefore, the trace form
    \begin{align*}
        \Tr: \kk(V) \times \kk(V) &\rightarrow \KK \\
        (f_1,f_2) &\mapsto \Tr_{\kk(V)/\KK}(f_1f_2)
    \end{align*}
    is nondegenerate. Since the bilinear form $B:\kk(V)\times \kk(V) \rightarrow \KK$ given by $(f_1,f_2)\mapsto \Reynolds(f_1f_2)$ differs from the trace form by a scalar, it too is nondegenerate.

    Because the $\psi^\lambda_i\left(v^\lambda_j\right)$ form a basis for $\kk(V)$ over $\KK$ as $\lambda$ varies in $\Irr_\kk(G)$ and $i,j$ vary in $1,\dots,d_\lambda$, the symmetric matrix
    \[
    \begin{pmatrix}
        \Reynolds\left(\psi^{\nu}_{i}\left(v^{\nu}_{j}\right) \psi^{\lambda}_{\ell}\left(v^{\lambda}_{m}\right)\right)
    \end{pmatrix}_{\substack{\nu,\lambda\in\Irr_\kk(G) \\ 1\leq i,j\leq d_{\nu} \\ 1\leq \ell,m\leq d_{\lambda}}}
    \]
    represents $B$ with respect to this basis. Since $B$ is nondegenerate, this matrix is nonsingular. In particular, focusing on the columns with $\lambda = \mu$, $m=k$, and $\ell=1,\dots, d_\mu$, they must be $\KK$-linearly independent. Thus there must be $d_\mu$ corresponding rows, for some choices of $\nu, i, j$, so that the resulting $d_\mu \times d_\mu$ submatrix is nonsingular. We take $h_1,\dots,h_{d_\mu}$ to be $\psi^{\nu}_{i}\left(v^{\nu}_{j}\right)$ for these choices of $\nu, i, j$.
\end{proof}

Let $\KK_\mathrm{low}$ be the subfield of $\KK$ generated over $\kk$ by the invariant polynomials of degree at most $\max(\Dspan + D_I, 2\Dspan)$:
\[
\KK_\mathrm{low} := \kk\left(\kk[V]^G_{\leq \max(\Dspan + D_I, 2\Dspan)}\right)\subseteq \KK.
\]
The following lemma uses the matrix in Lemma~\ref{lem:nonsingular} to show that the $\KK$-coordinates of any element \eqref{eq:basis-for-leqDI}, when written with respect to our basis \eqref{eq:basis-for-Vreg} for $\Vreg$, are contained in $\KK_\mathrm{low}$. This is the core of the argument. (It will eventually lead us to the conclusion that $\KK_\mathrm{low}=\KK$, but we of course do not assume this.)

Recall that $S$ is an index set for the irreducible subrepresentations $W_s$ of $\kk[V]_{\leq D_I}=\bigoplus_{s\in S} W_s$, and thereby also for the embeddings $\phi_s:V_{\lambda(s)}\rightarrow W_s \subseteq \kk[V]$.

\begin{lemma}\label{lem:matrix-equation}
    Take any $s\in S$. For brevity, write $\lambda := \lambda(s)$, and take any $j=1,\dots,d_\lambda$. Then we have a $\KK$-linear relation 
    \begin{equation}\label{eq:lc}
    \phi_s\left(v^{\lambda}_j\right) = a_1\psi_1^{\lambda}\left(v^{\lambda}_j\right) + \dots + a_{d_{\lambda}} \psi_{d_{\lambda}}^{\lambda}\left(v^{\lambda}_j\right)
    \end{equation}
    in $\kk(V)$, and in fact, the coefficients $a_1,\dots,a_{d_{\lambda}}$ belong to $\KK_\mathrm{low}$.
\end{lemma}

\begin{proof}
    Note that $\phi_s$ belongs to $\Hom_{\kk G}(V_\lambda,\kk(V))$. By observation~\ref{obs:K-basis}, $\psi_1^\lambda,\dots,\psi_{d_\lambda}^\lambda$ form a $\KK$-basis of $\Hom_{\kk G}(V_\lambda,\kk(V))$. Therefore, there exists a $\KK$-linear relation
    \[
    \phi_s = a_1\psi_1^\lambda + \dots + a_{d_\lambda}\psi^\lambda_{d_\lambda}
    \]
    in the $\KK$-vector space $\Hom_{\kk G}(V_\lambda,\kk(V))$. We obtain 
    \eqref{eq:lc} upon application of both sides to $v_j^\lambda$. Our work is to show that the $a_1,\dots,a_{d_\lambda}$, which a priori belong to $\KK$, actually belong to $\KK_\mathrm{low}$.

    We apply Lemma~\ref{lem:nonsingular} with $\mu = \lambda$ and $k=j$, obtaining $h_1,\dots, h_{d_\lambda}\in \kk[V]$ homogeneous of degree $\leq \Dspan$ so that the matrix 
    \[
    \left(\Reynolds\left(h_i\psi_\ell^\lambda\left(v^\lambda_j\right)\right)\right)_{1\leq i,\ell\leq d_\lambda}
    \]
    is nonsingular. Multiplying \eqref{eq:lc} through by each of $h_1,\dots,h_{d_\lambda}$ results in $d_\lambda$ equations. We then apply the Reynolds operator $\Reynolds$ to each of them. Noting that $\Reynolds$ is $\KK$-linear, we obtain the system of equations
    \begin{align*}
        \Reynolds\left(h_1\phi_s\left(v_j^\lambda\right)\right) &= a_1\Reynolds\left(h_1\psi_1^\lambda\left(v_j^\lambda\right)\right) + \dots + a_{d_\lambda}\Reynolds\left(h_1\psi_{d_\lambda}^\lambda\left(v^\lambda_j\right)\right) \\
         &\;\;\vdots \\
        \Reynolds\left(h_{d_\lambda}\phi_s\left(v_j^\lambda\right)\right) &= a_1\Reynolds\left(h_{d_\lambda}\psi_1^\lambda\left(v_j^\lambda\right)\right) + \dots + a_{d_\lambda}\Reynolds\left(h_{d_\lambda}\psi_{d_\lambda}^\lambda\left(v^\lambda_j\right)\right),
    \end{align*}
    or, in other words, the nonsingular matrix equation
    \begin{equation}\label{eq:matrix-eqn}
    \begin{pmatrix}\Reynolds\left(h_i\phi_s\left(v_j^\lambda\right)\right)\end{pmatrix}_i = \left(\Reynolds\left(h_i\psi_\ell^\lambda\left(v^\lambda_j\right)\right)\right)_{i\ell}\begin{pmatrix}a_\ell\end{pmatrix}_{\ell},
    \end{equation}
    where $i$ and $\ell$ both vary in $1,\dots,d_\lambda$. Solving this matrix equation for $\begin{pmatrix} a_\ell\end{pmatrix}_\ell$ expresses $a_1,\dots,a_{d_\lambda}$ rationally in terms of the other matrix entries. We now show that the latter all belong to $\KK_\mathrm{low}$.
    
    Indeed, each $h_i$ is a (homogeneous) polynomial of degree at most $\Dspan$, as is each $\psi_\ell^\lambda\left(v_j^\lambda\right)$. Meanwhile, $\phi_s\left(v_j^\lambda\right)$ is a (homogeneous) polynomial of degree at most $D_I$. Because $\Reynolds$ preserves degree, we conclude all the matrix entries in \eqref{eq:matrix-eqn} besides the $a$'s, are polynomials of degree at most $\max(\Dspan + D_I, 2\Dspan)$. So they belong to $\KK_\mathrm{low}$ by construction of the latter. Therefore, $a_1,\dots,a_{d_\lambda}$ belong to $\KK_\mathrm{low}$ as well.
\end{proof}

\subsection{Assembling the proof}\label{sec:assemble-proof}

In this subsection we assemble the proof of Theorem~\ref{thm:main} from the lemmas in Sections~\ref{sec:generic-orbit-ideal} and \ref{sec:matrix-eqs}. 
Both of the assumptions of Section~\ref{sec:matrix-eqs} on $\kk$, i.e., that it is algebraically closed  and of characteristic not dividing $|G|$, are in effect here.

To begin with, it is sufficient to prove that $\KK_\mathrm{low}=\KK$, because 
\[
\max(\Dspan + D_I, 2\Dspan)\leq 2\Dspan + 1
\]
by Lemma~\ref{lem:DI-and-Dspan}.

By Lemma~\ref{lem:coeffs-generate}, the $\KK$-coefficients that appear in any generating set for $I$ will generate $\KK$ as a field. We apply this lemma to a generating set consisting of elements of $I$ of degree $\leq D_I$ (which exists by definition of $D_I$). Our task is to show that the coefficients of these polynomials can be taken to belong to $\KK_\mathrm{low}$.

Let $f$ be any element of $I$ of degree $\leq D_I$. By definition, the coefficients that appear in $f$ are the $\KK$-coordinates of $f$ relative to the monomial basis for $\KK[X_1,\dots,X_n]_{\leq D_I}$, i.e., the basis consisting of monomials in $X_1,\dots, X_n$. But the basis \eqref{eq:basis-for-leqDI} for $\KK[X_1,\dots,X_n]_{\leq D_I}$ differs from the monomial basis by a change-of-basis matrix defined over $\kk$, as both the monomial basis and \eqref{eq:basis-for-leqDI} are actually bases for $\kk[X_1,\dots,X_n]_{\leq D_I}$. Thus, the $\KK$-coordinates of $f$ with respect to the monomial basis belong to $\KK_\mathrm{low}$ if and only if the $\KK$-coordinates of $f$ with respect to the alternative basis \eqref{eq:basis-for-leqDI} do so. We will be done if we can show that the $\KK$-coordinates of $f$ with respect to \eqref{eq:basis-for-leqDI} can be taken, possibly after multiplication through by an element of $\KK$, to belong to $\KK_\mathrm{low}$.

What it means for $f$ to belong to $I$ is that $f$ is a $\KK$-linear combination of the elements \eqref{eq:basis-for-leqDI}, viewed as elements of $\kk[X_1,\dots,X_n]_{\leq D_I}$, that, upon the substitution $X_i\mapsto x_i$ ($i=1,\dots,n$), evaluates to zero in $\kk(V)$. In other words, $f$ is a $\KK$-linear relation between the elements \eqref{eq:basis-for-leqDI} when they are viewed as belonging to $\kk(V)$. Viewing $\kk(V)$ as a finite-dimensional $\KK$-vector space with basis \eqref{eq:basis-for-Vreg}, we can compute any such linear relation using Gaussian elimination on the $\KK$-vectors that give the $\KK$-coordinates of the elements \eqref{eq:basis-for-leqDI} with respect to the $\KK$-basis \eqref{eq:basis-for-Vreg}. Lemma~\ref{lem:matrix-equation} tells us that all of these $\KK$-coordinates belong to $\KK_\mathrm{low}$. Since Gaussian elimination only involves field operations, the $\KK$-linear relations between the elements \eqref{eq:basis-for-leqDI} can be computed over $\KK_\mathrm{low}$, i.e., every $\KK$-linear relation is a $\KK$-multiple of a $\KK_\mathrm{low}$-linear relation. In particular $f$ can be taken (possibly after scaling by a $\KK$-factor) to be a linear relation between the elements \eqref{eq:basis-for-leqDI} with coefficients lying in $\KK_\mathrm{low}$.

Allowing $f$ to vary over a generating set for $I$ (and recalling Lemma~\ref{lem:coeffs-generate}), we conclude that elements of $\KK_\mathrm{low}$ generate $\KK$, and thus that $\KK_\mathrm{low}=\KK$. This completes the proof of Theorem~\ref{thm:main}. \qed

\section{Explicit examples}\label{sec:examples}

In this section, we consider examples. Section~\ref{ex:regular} shows that, in the non-modular setting, Theorem~\ref{thm:main} generalizes  \cite[Theorem~III.1]{edidin2025orbit}. Section~\ref{ex:worked-ex} illustrates in detail the inner workings of the proof of Theorem~\ref{thm:main}. Section~\ref{ex:sharp} exhibits that the inequality in Theorem~\ref{thm:main} is sharp.

\subsection{Representations containing the regular representation}\label{ex:regular}

It is known from signal processing that the orbits of the action of a compact Lie group on its regular representation $L^2(G,\CC)$ are generically separated by the {\em bispectrum}, which is a set of invariant functions of degree 3 made of holomorphic and antiholomorphic polynomials \cite{kakarala2009completeness, smach2008generalized}. When $G$ is finite, $L^2(G,\CC)$ is finite-dimensional; but because the entries of the bispectrum are not true polynomials, this orbit separation result does not immediately imply a field generation result. However, it does suggest the conjecture that $\bfield(G,V)\leq 3$ for any finite group $G$ and $V$ its regular representation. This was proven for finite abelian $G$ in non-modular characteristic in \cite[Theorem~4.1]{bandeira2023estimation}. (Equality occurs unless $|G|=2$.) Recently, Edidin and Katz \cite{edidin2025orbit}  generalized this result in the characteristic zero case, proving the conjecture and going further: if $G$ is any finite group, and $V$ merely {\em contains} the regular representation, then Theorem~III.1  of \cite{edidin2025orbit} implies that $\bfield(G,V) \leq 3$. 
    
In fact, \cite[Theorem~III.1]{edidin2025orbit} shows that the polynomial invariants of degree $\leq 3$ generically separate the orbits of $G$ on $V$; the result is formulated over any infinite field. In characteristic zero, this implies $\bfield(G,V)\leq 3$, by the following considerations. First, it does not affect $\bfield$ to base-change to the algebraic closure of $\kk$ \cite[Lemma~2.1]{blum2024degree}, so we may assume that $\kk$ is algebraically closed of characteristic zero. Then, the fact that $\kk[V]_{\leq 3}$ generically separates the orbits of $G$ on $V$ implies that it also generates the invariant field, by \cite[Lemma~2.1]{popov-vinberg}.\footnote{The fact that generic orbit separators over an algebraically closed field of characteristic zero generate the field of rational invariants can be deduced from a theorem  due to Rosenlicht \cite[Theorem~2]{rosenlicht1956some}.}

In the non-modular case, Theorem~\ref{thm:main} is a further generalization of \cite[Theorem~III.1]{edidin2025orbit}. Edidin and Katz' hypothesis that $V$ contains the regular representation implies that $D_\mathrm{span}=1$, as we now show.

\begin{proposition}\label{prop:reg-rep-implies-Dspan=1}
    Let $G$ be a nontrivial finite group and $\kk$ an arbitrary field. If $V$ is a representation of $G$ over $\kk$ containing the regular representation, then $\Dspan = 1$. 
\end{proposition}

\begin{proof}
    Since the regular representation is a projective and injective $\kk G$-module, $V$ contains it if and only if $V$ surjects onto it. Thus, by means of Theorem~\ref{thm:monotonic-in-the-rep} below, we immediately reduce the proof that $\Dspan \leq 1$ to the case that $V$ is itself the regular representation. That $\Dspan = 1$ will then follow, because $\Dspan$ must be at least $1$, since $\kk(V)$ is of dimension greater than $1$ as a $\kk(V)^G$-vector space, since $G$ is nontrivial.

    Let $x_1,\dots,x_n$ be  coordinate functions on $V$, with $n=|G|$. Enumerate the elements of $G$ by $g_1=1,g_2,\dots,g_n$. By definition of the regular representation, there exists $v\in V$ such that the $n$ vectors $g_1 v=v, g_2 v, \dots, g_n v$ are linearly independent. Thus the square matrix
    \[M(v):=
    \begin{pmatrix}
        x_1(v)&x_2(v)&\dots & x_n(v)\\
        x_1(g_2v)&x_2(g_2v)&\dots & x_n(g_2v)\\
        \vdots & \vdots & \ddots & \vdots \\
        x_1(g_nv)&x_2(g_nv)&\dots & x_n(g_nv)
    \end{pmatrix}
    =
    \begin{pmatrix}
        x_1(v)&x_2(v)&\dots & x_n(v)\\
        g_2^{-1}x_1(v)&g_2^{-1}x_2(v)&\dots & g_2^{-1}x_n(v)\\
        \vdots & \vdots & \ddots & \vdots \\
        g_n^{-1}x_1(v)&g_n^{-1}x_2(v)&\dots & g_n^{-1}x_n(v)
    \end{pmatrix}
    \]
    with entries in $\kk$ is nonsingular. It follows that the matrix
    \[
    M:=\begin{pmatrix}
        x_1&x_2&\dots & x_n\\
        g_2^{-1}x_1&g_2^{-1}x_2&\dots & g_2^{-1}x_n\\
        \vdots & \vdots & \ddots & \vdots \\
        g_n^{-1}x_1&g_n^{-1}x_2&\dots & g_n^{-1}x_n
    \end{pmatrix}
    \]
    with entries lying in $\kk[V]$, whose entry-by-entry evaluation at $v$ is the previous matrix $M(v)$, is also nonsingular. 

    This implies that $x_1,\dots,x_n$ are linearly independent over $\kk(V)^G$, as follows. A nontrivial linear relation 
    \[
    \sum_{j=1}^n c_jx_j = 0
    \]
    taking place in $\kk(V)$, with coefficients $c_j\in \kk(V)^G$, would also (by application of $g_1^{-1}=1,g_2^{-1},\dots,g_n^{-1}$) yield a nontrivial solution in the $c_j$'s (considering that they are $G$-invariant) to the matrix equation 
    \[
    \begin{pmatrix}
        x_1&x_2&\dots & x_n\\
        g_2^{-1}x_1&g_2^{-1}x_2&\dots & g_2^{-1}x_n\\
        \vdots & \vdots & \ddots & \vdots \\
        g_n^{-1}x_1&g_n^{-1}x_2&\dots & g_n^{-1}x_n
    \end{pmatrix}
    \begin{pmatrix}c_1\\c_2\\ \vdots \\ c_n
    \end{pmatrix}=0,
    \]
    which contradicts the nonsingularity of $M$.

    Since $x_1,\dots,x_n$ are linearly independent over $\kk(V)^G$, and $[\kk(V):\kk(V)^G]=n$, dimension counting leads to the conclusion that $x_1,\dots,x_n\in \kk[V]_1$ span $\kk(V)$ over $\kk(V)^G$. Therefore, $\Dspan\leq 1$, and, as above, strict inequality is impossible because $G$ is nontrivial. We conclude $\Dspan=1$.
\end{proof}

Proposition~\ref{prop:reg-rep-implies-Dspan=1} allows us to characterize the condition $\Dspan=1$.

\begin{proposition}\label{prop:characterize-Dspan=1}
    Let $G$ be a nontrivial group, $\kk$ an arbitrary field, and $V$ a finite-dimensional representation of $G$ over $\kk$. Then $\Dspan=1$ if and only if $\kk\oplus V$ contains the regular representation.
\end{proposition}

\begin{proof}
    By Proposition~\ref{prop:c-Dreg-Dspan-topdeg} below, $\Dreg\leq \Dspan$. Thus if $\Dspan = 1$, then $\Dreg \leq 1$, so the regular representation is contained in $\kk[V]_{\leq 1} = \kk \oplus V^*$. Then it is a homomorphic image of $\kk\oplus V$ by taking duals, and because it is also a projective $\kk G$-module this means it is contained in $\kk\oplus V$.

    In the other direction, suppose $\kk \oplus V$ contains the regular representation. Let $V':=\kk \oplus V$. By Proposition~\ref{prop:reg-rep-implies-Dspan=1} we have $\Dspan(G,V')=1$. Meanwhile, let $x'$ be a coordinate function on the $\kk$ summand of $V'$, let $x_1,\dots,x_n$ be coordinates on $V$, and let $\FF:=\kk(x')$. Then $\kk(V') = \kk(x',x_1,\dots,x_n)=\kk(x')(x_1,\dots,x_n) = \FF(V_\FF)$, and this identification is equivariant so $\kk(V')^G = \FF(V_\FF)^G$. Also, $\kk[V']_{\leq d} = \kk[x',x_1,\dots,x_n]_{\leq d} \subseteq \kk(x')[x_1,\dots,x_n]_{\leq d} = \FF[V_\FF]_{\leq d}$ for any $d$. Therefore, $\Dspan(G,V_\FF)\leq\Dspan(G,V')$, and we have
    \[
    \Dspan(G,V)=\Dspan(G,V_\FF)\leq \Dspan(G,V')=1,
    \]
    where the first equality is by Lemma~\ref{lem:base-change},  the inequality is by what has been done above, and the final equality is by Proposition~\ref{prop:reg-rep-implies-Dspan=1}. It follows (as in the proof of Proposition~\ref{prop:reg-rep-implies-Dspan=1}) that $\Dspan(G,V)=1$ because $G$ is nontrivial.
\end{proof}

\begin{remark}
    If $\Char\kk$ divides $|G|$, then $\kk$ is not a summand of the regular representation, so in that situation Proposition~\ref{prop:characterize-Dspan=1} becomes the statement that $\Dspan=1$ iff $V$ contains the regular representation.
\end{remark}

Thus, in the situation of Edidin and Katz' result, Proposition~\ref{prop:reg-rep-implies-Dspan=1} (or Proposition~\ref{prop:characterize-Dspan=1}) yields that $\Dspan=1$, so in the presence of the non-modular hypothesis, Theorem~\ref{thm:main} then specializes to $\bfield \leq 2(1) + 1 = 3$. Edidin and Katz' original conclusion of generic orbit separation then also follows: As mentioned in Section~\ref{sec:generic-orbit-ideal}, because $G$ is finite, the categorical quotient variety $V\sslash G$ is a geometric quotient, and $\kk(V)^G = \kk(V\sslash G)$ is the field of rational functions on this quotient. So the fact that generators for the rational function field of a variety separate generic points of the variety (together with $\bfield \leq 3$) gives us immediately that the invariants of degree $\leq 3$ separate generic orbits of $G$ on $V$.\footnote{The finiteness of $G$ makes this argument expedient, but is not really necessary for the conclusion that generators for the invariant field separate generic orbits. Rosenlicht's theorem \cite[Theorem~2]{rosenlicht1956some} (and see also \cite[Theorem~6.2]{dolgachev2003lectures}) says that,  for any rational action by an algebraic group $G$ on an irreducible algebraic variety $X$, there is a nonempty Zariski-open subset $U\subseteq X$ stable under the action of $G$, such that $U\sslash G$ is a geometric quotient and $\kk(X)^G=\kk(U\sslash G)$.}

\begin{remark}
    It is worth noting that $\Dspan=1$ is not a necessary condition for $\bfield \leq 3$. Here is an example.
    Let $G$ be the dihedral group of order $2n$. Let it act on $V=\CC^n$ as a permutation group, by the permutations that either preserve or reverse a given cyclic order on the standard basis. Note that this representation has only half the dimension of the regular representation, so $\Dspan > 1$ by dimension-counting. It is shown in \cite[Theorem~2.6]{edidin2025generic} that $\CC[V]^G_{\leq 3}$ separates generic orbits, and it follows by \cite[Lemma~2.1]{popov-vinberg} (just as above) that $\bfield \leq 3$.
\end{remark}

\subsection{Illustration of proof of the main theorem}\label{ex:worked-ex}
Take $\kk = \CC$. Let $G = Q_8 = \langle \bfi, \bfj\rangle$ be the quaternion group, and consider its irreducible 2-dimensional representation $V=\CC^2$ given by
\[
\bfi \mapsto \begin{pmatrix} & -1\\1 & \end{pmatrix},\; \bfj \mapsto \begin{pmatrix}-i& \\ & i\end{pmatrix}.
\]
Then $\bfi x = y$, $\bfi y = -x$, $\bfj x = ix$, and $\bfj y = -iy$. The invariant algebra $\CC[x,y]^{Q_8}$ is generated by the invariants $x^4+y^4$, $x^2y^2$ of degree 4, and $xy(x^4-y^4)$ of degree 6. The degree 6 invariant is a quadratic surd over the field generated by the degree 4 invariants. There are no invariants of degree lower than 6 outside the span of the degree 4 invariants, so $\bfield(Q_8,V) = \beta(Q_8,V)=6$ in this situation. We will see below that $\Dspan(Q_8,V) = 3$.

We now decompose the low-degree part of $\CC[x,y]$ into irreducible representations of $G$ in order to see the proof of Theorem~\ref{thm:main} in action. The group $G$ has five isomorphism classes of irreducible representations: the trivial representation, the representation $V$ itself, and three one-dimensional representations with kernels the three index-2 subgroups generated respectively by $\bfi$, $\bfj$, and $\bfk = \bfi\bfj$. We choose the index set
\[
\Irr_\CC(G) = \{\mathbf{1},\bfi,\bfj,\bfk,\mathbf{Sta}\}
\]
($\mathbf{Sta}$ for ``Standard"). We fix the following representatives of each isomorphism class of irreducible representations, and the following distinguished bases for them:
\[
V_\mathbf{1} = \left\langle v^\mathbf{1}\right\rangle,\; \bfi v^\mathbf{1}= \bfj v^\mathbf{1} = v^\mathbf{1}
\]
\[
V_\bfi = \left\langle v^\bfi \right\rangle, \; \bfi v^\bfi = v^\bfi, \; \bfj v^\bfi = -v^\bfi
\]
\[
V_\bfj = \left\langle v^\bfj \right\rangle, \; \bfi v^\bfj = -v^\bfj, \; \bfj v^\bfj = v^\bfj
\]
\[
V_\bfk = \left\langle v^\bfk \right\rangle, \; \bfi v^\bfk = -v^\bfk, \; \bfj v^\bfk = -v^\bfk
\]
\begin{align*}
V_\mathbf{Sta} = \left\langle v_x^\mathbf{Sta}, v_y^\mathbf{Sta}\right\rangle,\; \bfi v_x^\mathbf{Sta} &= v_y^\mathbf{Sta}, \; \bfi v_y^\mathbf{Sta} = - v_x^\mathbf{Sta},\\ \bfj v_x^\mathbf{Sta} &= iv_x^\mathbf{Sta},\; \bfj v_y^\mathbf{Sta} = -iv_y^\mathbf{Sta}
\end{align*}
To minimize visual clutter, we have suppressed subscripts for the $v^\lambda$ for those $\lambda$ for which $d_\lambda=1$, and we do the same below for the $\psi^\lambda$. Also, we have chosen the basis for $V_\mathbf{Sta}$ (along with its notation, in particular the subscripts) so that the map $v_x^\mathbf{Sta}\mapsto x$, $v_y^\mathbf{Sta} \mapsto y$ is an isomorphism of representations from $V_\mathbf{Sta}$ to $\CC[x,y]_1$. We proceed to decompose the rest of the low-degree components of $\CC[x,y]$ into irreducible representations of $Q_8$. It turns out that  $\Dspan=3$ while $D_I=4$ in this case (this will be justified below), so we go up to degree 4.
\begin{align*}
\CC[x,y]_0 &= \langle 1\rangle \cong V_\mathbf{1}\\
\CC[x,y]_1 &= \langle x,y \rangle \cong V_\mathbf{Sta}\\
\CC[x,y]_2 &= \langle x^2 + y^2\rangle\oplus \langle xy \rangle \oplus \langle x^2 - y^2\rangle \cong V_\bfi \oplus V_\bfj \oplus V_\bfk\\
\CC[x,y]_3 &= \langle y^3, -x^3\rangle \oplus \langle x^2y, -xy^2 \rangle \cong V_\mathbf{Sta} \oplus V_\mathbf{Sta}\\
\CC[x,y]_4 &= \langle x^4+y^4\rangle \oplus \langle x^2y^2 \rangle \oplus \langle xy(x^2-y^2)\rangle \oplus \langle x^4-y^4 \rangle \oplus \langle xy(x^2+y^2)\rangle \cong 2V_\mathbf{1} \oplus V_\bfi \oplus V_\bfj \oplus V_\bfk.
\end{align*}
The basis given for each subrepresentation of $\CC[x,y]$ above also specifies an embedding $\phi_s:V_{\lambda(s)} \rightarrow \CC[x,y]$ of the corresponding representative irreducible $V_\lambda$, by mapping the latter's distinguished basis to the specified basis in $\CC[x,y]$; we index these embeddings by the first of the listed basis elements in $\CC[x,y]$. For example, the line  
\[
\CC[x,y]_3 = \langle y^3, -x^3\rangle \oplus \langle x^2y, -xy^2 \rangle \cong V_\mathbf{Sta} \oplus V_\mathbf{Sta}
\]
specifies two embeddings $\phi_{y^3}, \phi_{x^2y}:V_\mathbf{Sta}\rightarrow \CC[x,y]_3$, namely
\[
\phi_{y^3}\left(v_x^\mathbf{Sta}\right) = y^3,\; \phi_{y^3}\left(v_y^\mathbf{Sta}\right) = -x^3
\]
and
\[
\phi_{x^2y}\left(v_x^\mathbf{Sta}\right) = x^2 y,\; \phi_{x^2y}\left(v_y^\mathbf{Sta}\right) = -xy^2.
\]
The other lines should be read similarly. So, in the notation of the proof of Theorem~\ref{thm:main},
\[
S=\{1, x, x^2+y^2, xy, x^2 - y^2, y^3, x^2y, x^4+y^4, x^2y^2, xy(x^2-y^2), x^4-y^4, xy(x^2+y^2)\},
\]
(the index set of these embeddings), and
\[
\lambda(1) = \mathbf{1},\; \lambda(x) = \mathbf{Sta},\; \lambda(x^2 + y^2) = \bfi,\; \lambda(xy) = \bfj,\; \text{ etc.}
\]
We have thus written down the basis $\phi_s(v_j^{\lambda(s)})$ for $\CC[x,y]_{\leq D_I}$; note that there are 15 elements. (There are $12$ elements of $S$, but three of the $\phi_s$'s have two-dimensional image.)

We note that, in so doing, we have also collected enough irreducible representations to constitute a regular representation: $\CC[x,y]_0\oplus \CC[x,y]_2$ plus any two of the three copies of $V_\mathbf{Sta}$ yield a regular representation. We (arbitrarily) pick the two  in degree 3, and show that the resulting regular representation even spans $\CC(x,y)$ over $\CC(x,y)^{Q_8}$.  Indeed,  this pair of embeddings $\phi_{y^3}, \phi_{x^2y}$ of $V_\mathbf{Sta}$ is linearly independent over $\KK$: this is manifested in the nonsingularity of the matrix
\[
\begin{pmatrix}
    y^3 & x^2y \\
    -x^3 & -xy^2
\end{pmatrix},
\] 
where the columns are the images of the distinguished basis for $V_\mathbf{Sta}$ under the two embeddings.  
Since $\dim V_\mathbf{Sta}=2$, it follows that the images of $\phi_{y^3},\phi_{x^2y}$ span the $V_\mathbf{Sta}$-isotypic component of $\CC(x,y)$ over $\KK$; the rest of the irreducible representations are one-dimensional, and embed in $\CC[x,y]_0\oplus \CC[x,y]_2$, so $\CC[x,y]_0$ and $\CC[x,y]_2$ span the rest of the isotypic components. Hence $\Dspan = 3$ as claimed above. So we have 
\[
\Vreg = \langle 1\rangle \oplus \langle x^2 + y^2\rangle\oplus \langle xy \rangle \oplus \langle x^2 - y^2\rangle \oplus \langle y^3, -x^3\rangle\oplus \langle x^2y, -xy^2\rangle,
\]
and
\begin{align*}
\psi^\mathbf{1} &= \phi_1,\\
\psi^\bfi &= \phi_{x^2+y^2},\\
\psi^\bfj &= \phi_{xy},\\
\psi^\bfk &= \phi_{x^2-y^2},\\
\psi^\mathbf{Sta}_1 &= \phi_{y^3},\\
\psi^\mathbf{Sta}_2 &= \phi_{x^2y}
\end{align*}
give the basis of embeddings of the irreducible representations into $\Vreg$.

We now find a generating set for $I$ by finding $\KK = \CC(x,y)^{Q_8}$-linear relations between the $15$ elements of our $\CC$-basis for $\CC[x,y]_{\leq D_I}$, using the strategy of Lemmas~\ref{lem:nonsingular} and \ref{lem:matrix-equation}. The elements of this basis not already in $\Vreg$ are
\[
\phi_x\left(v^\mathbf{Sta}_x\right) = x \text{ and } \phi_x\left(v^\mathbf{Sta}_y\right) = y
\]
in degree 1, and the five elements
\[
x^4+y^4,\; x^2y^2,\; xy(x^2-y^2),\; x^4-y^4,\; xy(x^2+y^2)
\]
in degree 4. We now represent these seven elements of this basis for $\CC[x,y]_{\leq D_I}$ as $\KK$-linear combinations of our basis for $\Vreg$, using the strategy of Section~\ref{sec:matrix-eqs}, showing that all the coefficients in these linear combinations lie in $\KK_\mathrm{low}$.

The method is best illustrated by the computations for $\phi_x(v^\mathbf{Sta}_x)=x$ and $\phi_x(v^\mathbf{Sta}_y)=y$, because the corresponding representation $V_\mathbf{Sta}$ is not one-dimensional (thus the matrix equation of Lemma~\ref{lem:matrix-equation} involves an actual matrix). So we do this case first. Because $\psi_1^\mathbf{Sta}=\phi_{y^3}$ and $\psi_2^\mathbf{Sta}=\phi_{x^2y}$ are a $\KK=\CC(x,y)^{Q_8}$-basis for $\Hom_{\CC G}(V_\mathbf{Sta},\CC(x,y))$, we know that $\phi_x$ lies in their $\KK$-span; i.e., there is a linear relation
\[
\phi_{x} = a_1\phi_{y^3} + a_2 \phi_{x^2y}
\]
with $a_1,a_2\in \KK$. We will compute $a_1,a_2$, showing that they are actually in $\KK_\mathrm{low}$. Applying both sides to $v^\mathbf{Sta}_x$, we have
\[
x = a_1y^3 + a_2x^2y.
\]
Now we find $h_1, h_2$ such that the matrix
\[
\begin{pmatrix}
\Reynolds (h_1y^3) & \Reynolds(h_1x^2y) \\
\Reynolds(h_2y^3) & \Reynolds(h_2x^2y)
\end{pmatrix}
\]
is nonsingular. We can take $h_1=xy^2$ and $h_2=x^3$. (Note that, because $V_\mathbf{Sta}$ is a self-dual representation, these choices illustrate the remark following Lemma~\ref{lem:nonsingular} that the the $h_i$ must be from the dual isotypic component.) The matrix equation is
\[
\begin{pmatrix}
    \Reynolds(xy^2 \cdot x) \\ \Reynolds(x^3 \cdot x)
\end{pmatrix} =
\begin{pmatrix}
\Reynolds (xy^2 \cdot y^3) & \Reynolds(xy^2 \cdot x^2y) \\
\Reynolds(x^3 \cdot y^3) & \Reynolds(x^3 \cdot x^2y)
\end{pmatrix}
\begin{pmatrix}a_1 \\ a_2\end{pmatrix},
\]
where, by construction, all the matrix entries on the left side have degree $\leq \Dspan + D_I$, while all those on the right side have degree $\leq 2\Dspan$; thus the solutions $a_1,a_2$ lie in $\KK_\mathrm{low}$. Computing the Reynolds operator, we get
\[
\begin{pmatrix}
    x^2y^2 \\ \frac{x^4+y^4}{2}
\end{pmatrix} =
\begin{pmatrix}
-\frac{xy(x^4-y^4)}{2} & 0 \\
0 & \frac{xy(x^4-y^4)}{2}
\end{pmatrix}
\begin{pmatrix}a_1 \\ a_2\end{pmatrix},
\]
so
\begin{align*}
    a_1 &= -\frac{2xy}{x^4-y^4} \\
    a_2 &= \frac{x^4+y^4}{xy(x^4-y^4)}
\end{align*}
are the solutions to
\[
\phi_x = a_1 \psi_1^\mathbf{Sta} + a_2 \psi_2^\mathbf{Sta}.
\]
Evaluating this latter equation on $v_x^\mathbf{Sta}$ and $v_y^\mathbf{Sta}$, we obtain the relations
\[
x = -\frac{2xy}{x^4-y^4}(y^3) + \frac{x^4+y^4}{xy(x^4-y^4)}(x^2y)
\]
and
\[
y = -\frac{2xy}{x^4-y^4}(-x^3) + \frac{x^4+y^4}{xy(x^4-y^4)}(-xy^2)
\]
holding in $\kk(V)$. These yield elements
\[
X + \frac{2xy}{x^4-y^4} Y^3 - \frac{x^4+y^4}{xy(x^4-y^4)} X^2Y
\]
and
\[
Y - \frac{2xy}{x^4-y^4}X^3 + \frac{x^4+y^4}{xy(x^4-y^4)} XY^2
\]
of $I$ lying in $\KK[X,Y]_{\leq D_I}$.

The computations for the degree-4 elements of our basis for $\CC[x,y]_{\leq D_I}$ are simpler (and perhaps not quite as illuminating), because the corresponding irreducible representations of $Q_8$ are one-dimensional. We illustrate with $\phi_{x^4-y^4}(v^\bfj)=x^4-y^4$; the other calculations are similar. This element, being in the $V_\bfj$-isotypic component, must lie in the (one-$\KK$-dimensional) image of $\psi^\bfj$, which is spanned by $xy$. We have
\[
x^4-y^4 = axy,
\]
where $a\in \KK$ is invariant. One can read off of this that $a = (x^4-y^4)/xy$, but the point is to prove that this $a$ lies in $\KK_\mathrm{low}$, so we multiply the equation through by $h$ in the dual component to $V_\bfj$ in $\Vreg$, so that the equation for $a$ is in terms of invariant polynomials. Now $V_\bfj$ is self-dual (actually, all the irreducible representations of $Q_8$ are self-dual), so we take $h=xy$, yielding
\[
xy(x^4-y^4)=(x^2y^2)a,
\]
and $a$ is manifested as a ratio of invariant polynomials of degree $\leq \max(\Dreg + D_I,2\Dreg)$ as desired. (Note that, whereas in the $>1$-dimensional situation we needed to apply the Reynolds operator after multiplying by the $h$'s to get the matrix entries to be invariant, for one-dimensional representations this isn't necessary.) The corresponding element of $I$ is 
\[
X^4 - Y^4 - \frac{x^4-y^4}{xy}XY,
\]
and the point of the above is that the coefficient lies in $\KK_\mathrm{low}$. By similar logic, from the rest of the degree 4 elements of our basis for $\CC[x,y]_{\leq D_I}$, we get the following elements of $I$, with confidence that the $\KK$-coefficients all actually lie in $\KK_\mathrm{low}$:
\[
X^4 + Y^4 - x^4+y^4,
\]
\[
X^2Y^2 - x^2y^2,
\]
\[
XY(X^2-Y^2) - \frac{xy(x^2-y^2)}{x^2+y^2}(X^2+Y^2),
\]
and
\[
XY(X^2+Y^2) - \frac{xy(x^2+y^2)}{x^2-y^2}(X^2-Y^2).
\]

Since the relations we have found express, as $\KK$-linear combinations of our $\CC$-basis for $V_\mathrm{reg}$, all the elements of our $\CC$-basis for $\CC[x,y]_{\leq D_I}$ that do not already lie in $V_\mathrm{reg}$, they actually constitute a $\KK$-spanning set of $\KK$-linear relations among the members of our basis for $\CC[x,y]_{\leq D_I}$, and so the corresponding elements of $I$ in fact generate $I$. Thus, by Lemma~\ref{lem:coeffs-generate}, the coefficients (which we know lie in $\KK_\mathrm{low}$) generate the field $\KK=\CC(x,y)^{Q_8}$, as desired. In fact, the $a_1,a_2$ computed above for the two degree-3 elements of $I$, resulting from the relation $\phi_x = a_1\psi_1^\mathbf{Sta} + a_2\psi_2^\mathbf{Sta}$, already generate $\KK$ by themselves. One can see this ``from scratch" by expressing the algebra generators for $\CC[x,y]^{Q_8}$ in terms of $a_1,a_2$:
\begin{align*}
    x^2y^2 &= \frac{1}{a_2^2 - a_1^2},\\
    x^4+y^4 &= \frac{2a_2}{a_1^3 - a_1a_2^2},\\
    xy(x^4-y^4) &= \frac{2}{a_1^3 - a_1a_2^2}.
\end{align*}

\begin{remark}
The reader has taken us at our word that $D_I=4$, so that by writing down a $\CC$-basis for $\CC[x,y]_{\leq 4}$, we have in fact written one down for $\CC[x,y]_{\leq D_I}$. One can verify this by computing a reduced Gr\"obner basis for the ideal $I'$ in $\KK[X,Y]$ generated by the seven above-identified elements of $I$ (of degrees 3 and 4). For example, with respect degree-lexicographic order with $X\succ Y$, we obtain 
\[
Y^4 + \frac{x^4-y^4}{2xy}XY - \frac{x^4+y^4}{2},
\]
\[
XY^3 - \frac{2x^3y^3}{x^4-y^4} X^2 + \frac{xy(x^4+y^4)}{x^4-y^4} Y^2,
\]
\[
X^3 - \frac{x^4+y^4}{2x^2y^2}XY^2 - \frac{x^4-y^4}{2xy} Y,
\]
and
\[
X^2Y - \frac{2x^2y^2}{x^4+y^4}Y^3 - \frac{xy(x^4-y^4)}{x^4+y^4} X.
\]
One deduces from the initial ideal $\langle Y^4,XY^3,X^3,X^2Y\rangle$ that $\KK[X,Y]/I'$ is  $8$-dimensional as a $\KK=\kk(V)^{Q_8}$-vector space. Since this is also true of $\KK[X,Y]/I\cong \kk(V)$ because $|Q_8|=8$, and $I'\subseteq I$ by construction, we can conclude $I'=I$, and thus $D_I=4$ as claimed.
\end{remark}

\subsection{Theorem~\ref{thm:main} is sharp}\label{ex:sharp}
Section~\ref{ex:regular} manifests Theorem~\ref{thm:main} as sharp if $|G|\geq 3$. Let $V$ be the regular representation of $G$. Then $\Dspan = 1$ and so $\bfield \leq 3$ by Theorem~\ref{thm:main}. But meanwhile also $\bfield \geq 3$ by \cite[Section~4.3.3]{bandeira2023estimation} unless $G$ is an elementary abelian 2-group $(\ZZ/2\ZZ)^k$ for some $k\in\NN$---and actually, $\bfield \geq 3$ even in the latter case, unless $k=0$ or $1$, by diagonalizing the action and considering the lattice of invariant Laurent monomials (as in \cite{hubert-labahn} and \cite{blum2024degree}): for $k>1$, there exist relations $\chi + \xi + \rho = 0$ between three distinct nontrivial characters in the character group of $(\ZZ/2\ZZ)^k$, and these correspond to degree-3 invariant monomials not in the field generated by the degree $\leq 2$ invariants. So $\bfield = 3$ unless $|G|=1$ or $2$.

Here is another example, with arbitrarily large $\bfield$. Let $V=\CC^2$. For an odd natural number $n$, let $\zeta$ be a primitive $n$th root of unity in $\CC$. Let $G = C_n = \langle g\rangle$ be cyclic of order $n$ with generator $g$. Suppose $g$ acts on $V$ via the matrix
\[
\begin{pmatrix}
        \zeta & \\ 
         & \zeta^{-1}
\end{pmatrix}.
\]
Then if $x,y$ are coordinate functions on $\CC^2$, we have $gx = \zeta^{-1}x$, $gy = \zeta y$. A minimum-degree polynomial basis for $\CC(x,y)$ over $\CC(x,y)^G$ is
\[
x^{(n-1)/2},\dots, x, 1, y, \dots , y^{(n-1)/2}.
\]
Thus
\[
\Dspan = \frac{n-1}{2}
\]
here. Meanwhile, minimum-degree polynomial generators for $\CC(V)^G$ are $x^n$ and $xy$ (or $y^n$ and $xy$), so
\[
\bfield = n
\]
in this case, realizing equality in Theorem~\ref{thm:main}.

\section{The spanning degree}\label{sec:spanning-deg-bound}
In this section, we study the quantity $\Dspan$, giving connections to other quantities of interest, some monotonicity and subadditivity results, and some bounds. Most of the results have no hypotheses on the characteristic of $\kk$, so throughout this section, $\kk$ represents a field of arbitrary characteristic unless explicitly stated otherwise.

\subsection{Relation to $\topdeg$, $\Dirr$, and $\Dreg$}\label{sec:Dirr-Dreg-Dspan-topdeg}

The introduction motivated the study of $\Dspan$ with its relation to other quantities studied in invariant and representation theory, specifically:
\begin{enumerate}
    \item $\topdeg(G,V)$, the minimum degree $d$ such that the vector space $\kk[V]_{\leq d}$ generates $\kk[V]$ as a module over $\kk[V]^G$.
    \item $\Dreg(G,V)$, the minimum degree $d$ such that the vector space $\kk[V]_{\leq d}$ contains a copy of the regular representation of $G$.
    \item $\Dirr(G,V)$, the minimum $d$ so that the tensor powers of order up to the $d$th contain every irreducible representation of $G$ as a composition factor.
\end{enumerate}
As with $\Dspan$ and $\bfield$, we drop the $(G,V)$ when they are clear from context. We now justify the claims made in the introduction about how these quantities are related.

\begin{proposition}\label{prop:c-Dreg-Dspan-topdeg}
Let $G$ be a finite group, $\kk$ a field, and $V$ a finite-dimensional representation of $G$.
\begin{enumerate}
    \item We have
    \[
    \Dirr \leq \Dreg \leq \Dspan \leq \topdeg.
    \]\label{part:general-ineq}
    \item If it happens that $G$ is abelian and the characteristic of $\kk$ does not divide the order of $G$, then in fact
    \[
    \Dirr = \Dreg = \Dspan.
    \]\label{part:abelian-ineq}
\end{enumerate}
\end{proposition}

\begin{proof}
    Part~\ref{part:general-ineq}: For the first inequality, because the regular representation is self-dual, and both projective and injective as a module over the group ring $\kk G$, it follows that for any $d$, the $G$-representation $\kk[V]_{\leq d}$ contains the regular representation if and only if its dual does so. In particular, $(\kk[V]_{\leq \Dreg})^*$ contains the regular representation. But
    \[
    (\kk[V]_{\leq \Dreg})^* = \bigoplus_{j=0}^{\Dreg} \Sym^j(V),
    \]
    and in turn,
    \[
    \bigoplus_{j=0}^{\Dreg} V^{\otimes j}
    \]
    surjects onto the latter. Since every irreducible representation of $G$ occurs as a composition factor of the regular representation, pulling back along this surjection we find every irreducible representation of $G$ as a composition factor of $\bigoplus_{j=0}^{\Dreg} V^{\otimes j}$. It follows that $\Dirr\leq \Dreg$.

    For the second inequality, recall the notation $\KK = \kk(V)^G$ and the surjection $\Xi:\KK[X_1,\dots,X_n]\rightarrow \kk(V)$ from Section~\ref{sec:main-result} (see equation \eqref{eq:Xi-X}). By definition of $\Dspan$, the restriction of  $\Xi$ to $\KK[X_1,\dots,X_n]_{\leq \Dspan}$ is still a surjection onto $\kk(V)$; the latter is a  projective $\KK G$-module (because $\kk(V)\cong \KK G$ as $\KK G$-modules). Therefore, the restriction of $\Xi$ to $\KK[X_1,\dots,X_n]_{\leq \Dspan}$ splits, so $\KK[X_1,\dots,X_n]_{\leq \Dspan}$ has a summand isomorphic to $\KK G$. Then  a variant of the Noether-Deuring Theorem  \cite[Problem~3.8.4(ii)]{etingof2011introduction} asserts that $\kk G$ is isomorphic to a summand of $\kk[V]_{\leq \Dspan}$. It follows that $\Dreg\leq \Dspan$.\footnote{The astute reader may wonder why this proof, which finds the regular representation as a summand of $\kk[V]_{\leq \Dspan}$, is so much shorter than Section~\ref{sec:graded-normal-basis}, and gets by without the algebraic closure hypothesis in Proposition~\ref{prop:normal-basis}. The difficulty in the latter was not just in making sure the desired summand can be taken to be graded: if it is not graded, write $\kk[V]_{\leq \Dspan}=W\oplus Y$ for $W\cong \kk G$ the summand found by the proof, and then use the Krull-Remak-Schmidt Theorem to compare a decomposition of $\kk[V]_{\leq \Dspan}$ that refines the grading, to another one that refines the decomposition $W\oplus Y$. This shows $W$ is isomorphic to a sum of indecomposables that refines the grading. Rather, the key difficulty in Section~\ref{sec:graded-normal-basis} was to make sure that the summand $\Vreg$ found there {\em both} is graded and simultaenously spans $\kk(V)$ over $\KK$.}

    For the third and final inequality, note that  $\kk(V) \cong \kk[V]\otimes_{\kk[V]^G} \kk(V)^G$ because $G$ is finite. By definition of $\topdeg$, the elements of $\kk[V]_{\leq \topdeg}$ generate $\kk[V]$ as a $\kk[V]^G$-module; it follows by base change from $\kk[V]^G$ to $\kk(V)^G$ that they also generate $\kk(V)$ as a $\kk(V)^G$-module. Thus $\Dspan\leq \topdeg$.

    Part~\ref{part:abelian-ineq}: We first argue that nothing is lost by base-changing to a field containing $|G|$th roots of unity. Let $\FF := \kk(\zeta)$ where $\zeta$ is a primitive $|G|$th root of unity in some extension of $\kk$, and let $V_\FF = V\otimes_\kk \FF$. (Note that such a $\zeta$ exists due to the non-modular hypothesis.)
    
    That $\Dspan$ is unaffected by base-change was proven in Lemma~\ref{lem:base-change}. For $\Dirr$, we argue as follows.  First note that $\FF$ is a splitting field for $G$ \cite[Corollary~9.15 and Theorem~10.3]{isaacs}, and every $\FF$-irreducible representation of $G$ is a constituent (=summand, because Part~\ref{part:abelian-ineq} is under the non-modular hypothesis) of the base change to $\FF$ of one of the $\kk$-irreducibles of $G$ \cite[Corollary~9.5(c)]{isaacs}. So if $\kk, V, V^{\otimes 2},\dots,V^{\otimes d}$ contain  every $\kk$-irreducible of $G$, then $\FF, V_\FF, (V_\FF)^{\otimes 2},\dots,(V_\FF)^{\otimes d}$ contain every $\FF$-irreducible of $G$  as well (using that $(V_\FF)^{\otimes j} = (V^{\otimes j})_\FF$). Thus $\Dirr(G,V_\FF) \leq \Dirr(G,V)$. In the other direction, the base change to $\FF$ of any $\kk$-irreducible representation of $G$ splits into a direct sum of {\em distinct} $\FF$-irreducibles (by combining \cite[Theorem~9.21(a) and (b)]{isaacs} and \cite[p.~93]{serre1977linear}),\footnote{To make this inference explicit: If $\kk$ has positive characteristic, the fact that the $\FF$-irreducible constituents of a $\kk$-irreducible representation are all distinct is automatic \cite[Theorem~9.21(a) and (b)]{isaacs}. If $\kk$ has characteristic zero, it holds because $G$ is abelian, and therefore the Schur index of each of its irreducible representations over the splitting field $\FF$, which counts its multiplicity in the $\kk$-irreducible that contains it, is 1  \cite[p.~93]{serre1977linear}. That said, this point makes the argument a little more economical but isn't really needed. The claim about $\Dirr$ and base change goes through without assuming $G$ is abelian. In characteristic 0, a similar and only slightly more involved argument to what follows goes through, using the fact that a sum of copies of a certain $\FF$-irreducible cannot be defined over $\kk$ unless the number of copies has reached its Schur index \cite[Corollary~10.2(d)]{isaacs}.} and any two nonisomorphic $\kk$-irreducibles split (after base-changing) into {\em disjoint} (even up to isomorphism) sets of $\FF$-irreducibles \cite[Corollary~9.7]{isaacs}; thus the (isomorphism classes of) $\FF$-irreducibles of $G$ yield a partition of the (isomorphism classes of) $\kk$-irreducibles (after base changing to $\FF$). Since we are now assuming non-modularity, $\kk G$ and $\FF G$ are semisimple, so the isomorphism class of a representation is determined by its constituents (and their multiplicities). In particular, if some  $\FF$-irreducibles $X_1,\dots,X_r$ appear as the constituents of the base change to $\FF$ of a $\kk$-irreducible $W$, we must have
    \[
    X_1\oplus \dots\oplus X_r \cong W_\FF.
    \]
    Thus, if $\FF\oplus V_\FF\oplus(V_\FF)^{\otimes 2}\oplus \dots \oplus(V_\FF)^{\otimes d}$ contains a representative of every $\FF$-irreducible representation of $G$, by organizing them according to the $\kk$-irreducibles of which they are constituent, we find the base change to $\FF$ of (a representative of) each $\kk$-irreducible in $\FF\oplus V_\FF\oplus(V_\FF)^{\otimes 2}\oplus \dots \oplus(V_\FF)^{\otimes d}$. Thus (a representative of) every $\kk$-irreducible must be present in $\kk\oplus V\oplus V^{\otimes 2}\oplus\dots\oplus V^{\otimes d}$. So $\Dirr(G,V) \leq \Dirr(G,V_\FF)$, and therefore $\Dirr(G,V)=\Dirr(G,V_\FF)$.

    For $\Dreg$, if $\kk[V]_{\leq d}$ contains a $\kk G$-submodule isomorphic to $\kk G$, then  $\FF[V_\FF]_{\leq d} = \FF\otimes_\kk \kk[V]_{\leq d}$ contains an $\FF G$-submodule isomorphic to $\FF G = \FF \otimes_\kk \kk G$, so $\Dreg(G,V_\FF)\leq \Dreg(G,V)$. In the other direction, if $\FF[V_\FF]_{\leq d}$ contains an $\FF G$-submodule isomorphic to $\FF G$, then it has $\FF G$ as a summand, and the same variant of the Noether-Deuring theorem quoted above \cite[Problem~3.8.4(ii)]{etingof2011introduction} then tells us that $\kk[V]_{\leq d}$ has $\kk G$ as a summand. Thus $\Dreg(G,V)\leq \Dreg(G,V_\FF)$, whence equality.

    This established, we can without loss of generality assume that $\kk$ contains $|G|$th roots of unity, whereupon because $G$ is abelian, its action on $V$ can be diagonalized over $\kk$. We work in a diagonal basis $v_1,\dots, v_n$ for $V$. The dual basis $x_1,\dots,x_n$ for $V^*$ also receives a diagonal action. For $j=1,\dots,n$ let $\chi_j:G\rightarrow \kk^\times$ be the character by which $G$ acts on $x_j$, i.e., such that $gx_j = \chi_j(g)x_j$ for each $g\in G$. Note that then the inverse $\chi_j^{-1}$ in the character group $\widehat G:=\Hom(G,\kk^\times)$ describes the action of $G$ on $v_j$. 
    
    Again because $G$ is abelian, the elements $\chi \in \widehat G$ index a complete set of representatives for the isomorphism classes of $G$'s $\kk$-irreducible representations $W_\chi$, where $W_\chi = \langle w_\chi \rangle_\kk$ with the action given by $g w_\chi = \chi(g) w_\chi$ for all $g$, and the regular representation of $G$ over $\kk$ is isomorphic to $\bigoplus_{\chi\in \widehat G} W_\chi$.
    
    We first prove $\Dirr = \Dreg$. A basis for the tensor algebra $\bigoplus_{j\geq 0} V^{\otimes j}$ is given by the noncommutative monomials in the $v_1,\dots,v_n$, and a basis for $\kk[V]$ is given by the commutative monomials in the $x_1,\dots,x_n$. Either a commutative or a noncommutative monomial is an eigenvalue for the action of $G$, and the action is given by the product in $\widehat G$ of the characters describing the actions of $G$ on the individual $x_j$'s or $v_j$'s; the noncommutativity in the tensor algebra does not introduce any complication, because $\widehat G$ is commutative anyway. Explicitly,
    \[
    g(x_{j_1}x_{j_2}\dots x_{j_s}) = [(\chi_{j_1}\chi_{j_2}\dots\chi_{j_s})(g)](x_{j_1}x_{j_2}\dots x_{j_s})
    \]
    and
    \[
    g(v_{j_1}\otimes v_{j_2}\otimes \dots \otimes v_{j_s}) = [(\chi_{j_1}\chi_{j_2}\dots\chi_{j_s})^{-1}(g)](v_{j_1}\otimes v_{j_2}\otimes \dots \otimes v_{j_s}).
    \]
    Thus, the characters that index the irreducible constituents of $\kk[V]_{\leq d}$ are precisely the inverses of those that index the irreducible constituents of $\bigoplus_{j=0}^d V^{\otimes j}$, for any $d$. Since $\widehat G$ is stable under inversion, $\kk[V]_{\leq d}$ contains a representative of every $\kk$-irreducible exactly when $\bigoplus_{j=0}^d V^{\otimes j}$ does. Since the regular representation is the sum of each $\kk$-irreducible once, this means $\Dirr=\Dreg$.

    To obtain $\Dreg=\Dspan$, we need to show that if $\kk[V]_{\leq d}$ contains a regular representation, then it spans $\kk(V)$ over $\kk(V)^G$. One can see this as follows. First, \cite[Lemma~2.2]{blum2024degree} shows that any set of monomials in the $x_j$'s that contains, for each character $\chi \in \widehat G$, one monomial receiving the $G$-action defined by that character, is also a complete set of coset representatives for the group of $G$-invariant Laurent monomials in the $x_j$'s inside the group of all Laurent monomials. Then, \cite[Lemma~2.5]{blum2024degree} yields that these same monomials form a $\kk(V)^G$-basis of $\kk(V)$.\footnote{Another approach to $\Dreg=\Dspan$ uses ideas from Section~\ref{sec:matrix-eqs}. Although we did not optimize the discussion in Section~\ref{sec:main-result} to draw this out, it follows from the ideas in Section~\ref{sec:matrix-eqs} that a regular representation found inside $\kk[V]_{\leq d}$ can fail to span $\kk(V)$ over $\kk(V)^G$ only if there are linear dependencies over $\kk(V)^G$ inside of the individual isotypic components of the regular representation in question. But this is impossible when $G$ is abelian because each isotypic component is one-dimensional over $\kk(V)^G$.}
\end{proof}

\begin{remark}
    Another nice case besides part~\ref{part:abelian-ineq} of Proposition~\ref{prop:c-Dreg-Dspan-topdeg} where some of the inequalities in part~\ref{part:general-ineq} become equalities is when $\kk=\CC$ and $G$ acts as a complex reflection group on $V$ (or, more generally, when $\kk$ has characteristic zero and $G$ is generated by pseudoreflections on $V$). In that case, $\Dreg=\Dspan=\topdeg$, as follows. The coinvariant algebra $\kk[V]_G$ is isomorphic as a $G$-representation to the regular representation (e.g., by \cite[Theorem~24-1]{kane2001reflection}), and is also a Poincar\'e duality algebra (by \cite[Theorem~1.5]{kane1994poincare}). In particular, the nonzero component of highest degree (say $D$) in $\kk[V]_G$ is a one-dimensional representation of $G$, and that representation only occurs with multiplicity one in $\kk[V]_G$, and therefore only in degree $D$. Meanwhile, because the Hilbert ideal $I_H:=\kk[V]^G_{>0}\kk[V]$, which is the kernel of the canonical surjection $\kk[V]\rightarrow \kk[V]_G$, is generated by invariants of positive degree, it cannot contain a copy of the one-dimensional representation $X:=(\kk[V]_G)_D$ in some degree component $(I_H)_d$ unless $\kk[V]_{<d}$ already contains a copy of $X$. In particular, the lowest-degree occurrence of $X$ in $\kk[V]$ does not lie in the kernel of the canonical surjection $\kk[V]\rightarrow \kk[V]_G$, so $X$ must occur in the same degree in $\kk[V]_G$. It follows that $X$ cannot occur anywhere in $\kk[V]_{<D}$. Thus $\Dreg = D$. Since clearly $\topdeg = D$ as well, this implies $\Dreg = \Dspan = \topdeg$, by Proposition~\ref{prop:c-Dreg-Dspan-topdeg} part~\ref{part:general-ineq}.
    
    But in general, all the inequalities in part~\ref{part:general-ineq} of Proposition~\ref{prop:c-Dreg-Dspan-topdeg} can be strict. Here are some examples illustrating the various phenomena:
    \begin{itemize}
        \item The inequality $\Dirr\leq \Dreg$ can be strict for two different reasons; we illustrate each.
        \begin{enumerate}
        \item {\em Multiplicity.} Suppose $\kk=\CC$. Let $V$ be what Edidin and Katz call the {\em complete multiplicity-free representation} of $G$ \cite[Example~IV.2]{edidin2025orbit}, i.e., the sum of all the irreducibles of $G$ taken once each. Then $\Dirr=1$, but if $G$ is not abelian then $V$ is smaller than the regular representation, so $\Dreg(G,V)>1$.

        Another example is if $\kk$ has characteristic $p$ and $G$ is a nontrivial $p$-group. In this case, the only irreducible representation of $G$ is the trivial representation; it occurs $|G|$ times as a composition factor of the regular representation. The trivial representation is already present in degree $0$, so $\Dirr = 0$. On the other hand, $\Dreg\geq 1$, with the inequality strict unless $V$ contains the regular representation.
        
        \item {\em The tensor algebra is bigger than the symmetric algebra.} Let $G$ be the dihedral group $D_3$ in its canonical action by real orthogonal rotation and reflection matrices on the plane $V=\CC^2$. This action makes $V$ an irreducible faithful representation. Then $V^{\otimes 2}$ can be seen as the space of $2\times 2$ complex matrices with the action of the original matrix group by conjugation. This decomposes into three subrepresentations: the scalar matrices (the trivial representation), the traceless symmetric matrices (on which the action is isomorphic to $V$ itself), and the skew-symmetric matrices (this is the sign representation). Thus, all irreducibles of $G$ occur in $V^{\otimes 2}$ and thus in $\CC \oplus V \oplus V^{\otimes 2}$, so $\Dirr=2$. On the other hand, $\Dreg(G,V)\geq 3$ because the sign representation does not occur anywhere in the space $\CC[V]_{\leq 2}$ of homogeneous polynomials of degree $\leq 2$. (It turns out that $\Dreg=3$ in this case.) In fact, $\CC[V]_{\leq 2}$ is isomorphic as a $G$-representation precisely to the sum of all the constituents of $\CC\oplus V \oplus V^{\otimes 2}$ {\em except} for the one corresponding to the skew-symmetric matrices.
        \end{enumerate}
        
        \item The inequality $\Dreg\leq \Dspan$ is strict when a regular representation is found in $\kk[V]_{\leq d}$ but it doesn't span $\kk(V)$ over $\kk(V)^G$ because of linear relations over $\kk(V)^G$. To illustrate, consider $A_4$, the alternating group on four points. This group has three one-dimensional characters, and a three-dimensional irreducible representation  as the rotational symmetries of a regular tetrahedron. If $V$ is the canonical permutation representation and $W$ is the three-dimensional irreducible representation, then $\CC[V]_1$ contains one copy of $W$ (and a trivial representation), while $\CC[V]_2$ contains two copies of $W$ and all three one-dimensional characters (including the trivial representation twice). Thus $\CC[V]_1\oplus \CC[V]_2$ contains a regular representation, and $\Dreg=2$. However, one of the copies of $W$ in degree 2 is obtained from the degree $1$ copy by multiplying by an element of the degree 1 trivial representation, so there is a linear dependence over $\kk(V)^G$ between them, and $\Dspan > 2$. (In fact, $\Dspan = 3$ in this case.)
        
        \item The inequality $\Dspan\leq \topdeg$ can be strict because generating $\kk[V]$ as a module over the ring $\kk[V]^G$ is a stricter demand on the set of elements of $\kk[V]_{\leq d}$ than generating $\kk(V)$ as a vector space over the field $\kk(V)^G$. For example the representation of $G=C_n$ given in Section~\ref{ex:sharp} (with the generator $g\in G$ acting by $gx = \zeta^{-1}x$, $gy=\zeta y$ with $\zeta$ a primitive $n$th root of unity) had $\Dspan=(n-1)/2$, but $\topdeg=n-1$ in that case because the monomials $x^{n-1},y^{n-1}$ are not in the submodule generated over $\CC[V]^G$ by the polynomials of degree less than $n-1$. Indeed, because $\CC[V]^G$ does not contain any monomials in $x$ or $y$ alone of degree less than $n$, one cannot get $x^{n-1}$ or $y^{n-1}$ from polynomials of lower degree by multiplication by elements of $\CC[V]^G$. On the other hand, when working with $\CC(V)$ as a vector space over $\CC(V)^G$, one can for example obtain $x^{n-1}$ as $(x^{n-1}/y)y$ since $x^{n-1}/y$ is in $\CC(V)^G$. 
        
        In general, essentially the same argument gives that whenever $G$ is cyclic and $V$ contains a one-dimensional faithful representation over $\CC$, then $\topdeg(G,V)=|G|-1$. On the other hand, in this same situation, $\Dspan(G,V)$ varies in an interesting and not-yet-well-understood way with the representation.
    \end{itemize}
\end{remark}

\subsection{Monotonicity and subadditivity results}\label{sec:monotonicity}

In this section we show that $\Dspan$ satisfies certain monotonicity and subadditivity results. The analogous monotonicity results fail for $\bfield$, but the subadditivity analogue goes through, with $\max$ replacing the sum.

\begin{definition}[{See~\cite[VIII.3]{lang2012algebra}}]
    Suppose that $L_1$ and $L_2$ are subfields of a field $M$ containing $\kk$. Then $L_1$ and $L_2$ are called {\em linearly disjoint} over $\kk$ if elements of $L_1$ are linearly dependent over $L_2$ if and only if they are linearly dependent over $\kk$.
\end{definition}
Linear disjointness is equivalent to the map
    $L_1\otimes_\kk L_2\to M$ given by 
    $\sum_{i} a_i\otimes b_i\mapsto \sum_i a_ib_i$ being injective.
From this it is clear that linear disjointness is symmetric in $L_1$ and $L_2$.
\begin{lemma}\label{lem:disjoint-invariants}
Suppose a group $G$ acts on a field $M$ by automorphisms, and $L$ is a subfield of $M$ that is setwise stable under the action of $G$. Then $L$ and $M^G$ are linearly disjoint over $L^G$.
\end{lemma}

\begin{proof}
Suppose that $a_1,\dots, a_n\in M^G$ are linearly dependent over $L$. There is an $m$ such that $a_1,\dots, a_{m-1}$ are linearly independent and $a_1,\dots, a_m$ are linearly dependent.
Then we can write $a_m=\sum_{i=1}^{m-1}a_ib_i$, where $b_1,\dots,b_{m-1}\in L$ are unique. If $g\in G$ then we have
$$
a_m=g\cdot a_m=\sum_{i=1}^{m-1}(g\cdot a_i)(g\cdot b_i)=\sum_{i=1}^{m-1}a_i(g\cdot b_i).
$$
Because $b_1,\dots,b_{m-1}\in L$ are unique, and $g\cdot b_1,\cdots, g\cdot b_{m-1}$ are in $L$, 
we have $b_i=g\cdot b_i$ for all $g\in G$ and all $i$. So $b_1,\dots,b_{m-1}\in L^G$ and $a_1,\dots,a_n$ are linearly dependent over $L^G$.
\end{proof}
\begin{theorem}[Monotonicity in the representation]\label{thm:monotonic-in-the-rep}
Suppose $G$ is a finite group over $\kk$.
If $V$ and $W$ are finite-dimensional representations of $G$ over $\kk$ such that $V$ is faithful and $W$ surjects onto it, then we have
$$
\Dspan(G,W)\leq \Dspan(G,V).
$$
\end{theorem}
\begin{proof}
Because $W$ surjects onto $V$, $V^*$ embeds as a $\kk G$-module in $W^*$, so $\kk(V)$ embeds as a $\kk G$-module in $\kk(W)$. We identify $\kk(V)$ with its image in $\kk(W)$. Then, by Lemma~\ref{lem:disjoint-invariants}, the subfields $\kk(V)$ and $\kk(W)^G$ are linearly disjoint over $\kk(V)^G$. It follows that any basis of $\kk(V)$ as a $\kk(V)^G$-vector space is also a basis of $\kk(V)\kk(W)^G$ as a $\kk(W)^G$-vector space, and we have
$$[\kk(V)\kk(W)^G:\kk(W)^G]=[\kk(V):\kk(V)^G]=|G|.
$$
Since $\kk(V)\kk(W)^G\subseteq \kk(W)$ and $[\kk(W):\kk(W)^G] = |G|$, we must have $\kk(V)\kk(W)^G=\kk(W)$.
Let $d=\Dspan(G,V)$. Then we have
$\kk[V]_{\leq d}\kk(V)^G=\kk(V)$ and
$$
\kk(W)=\kk(V)\kk(W)^G= \kk[V]_{\leq d}\kk(W)^G\subseteq
\kk[W]_{\leq d} \kk(W)^G.
$$
This proves that $\Dspan(G,W)\leq d$.
\end{proof}

\begin{remark}
    In the non-modular case, $W$ surjects onto $V$ if and only if $V$ embeds in $W$, if and only if $V$ is a summand of $W$, so the assumption that $W$ surjects onto $V$ in the hypothesis of Theorem~\ref{thm:monotonic-in-the-rep} can be replaced with the assumption that $V\subseteq W$.
\end{remark}

\begin{remark}
This monotonicity property fails for $\bfield$, see \cite[Example~3.10]{blum2024degree}.
\end{remark}

\begin{remark}
    The ring analogue to $\Dspan$, $\topdeg$, satisfies a monotonicity property almost opposite to the one for $\Dspan$ given in Theorem~\ref{thm:monotonic-in-the-rep}: $\topdeg(G,V)\leq \topdeg(G,W)$ if $V\subseteq W$, see \cite[Lemma~5]{kohls2014top}. On the other hand, $\topdeg(G,V\oplus W) \leq \topdeg(G,V)+\topdeg(G,W)$ \cite[Lemma~6]{kohls2014top}.
\end{remark}

\begin{lemma}[Monotonicity in the group]\label{lem:monotonic-in-the-group}
    Suppose that $V$ is a finite dimensional representation of a finite group $G$ over the field $\kk$ and $H\subseteq G$ is a subgroup. Then $V$ is a representation of $H$ by restriction, and we have $\Dspan(H,V)\leq \Dspan(G,V)$.
    \end{lemma}
    \begin{proof}
    Let $d=\Dspan(G,V)$. Then we have
    $$
\kk(V)=\kk[V]_{\leq d} \kk(V)^G\subseteq \kk[V]_{\leq d} \kk(V)^H\subseteq \kk(V),
$$
so $\kk(V)=\kk[V]_{\leq d}\kk(V)^H$ and $\Dspan(H,V)\leq d$.
    \end{proof}
\begin{remark}
This monotonicity property also does not hold for $\bfield$. If $n\geq 4$ and $V$ is the canonical permutation representation of $S_n$ on $V=\kk^n$, then we have $\bfield(A_n,V)={n\choose 2}>n=\bfield(S_n,V)$. In this case we have $\Dspan(S_n,V)=\Dspan(A_n,V)={n\choose 2}$.
\end{remark}

\begin{remark}
    An analogue to Lemma~\ref{lem:monotonic-in-the-group} is known for $\topdeg(G,V)$ \cite[Lemma~4]{kohls2014top}, and can be proven by the same technique.
\end{remark}

\begin{proposition}[Subadditivity]\label{prop:subadditivity}
    Suppose that $G$ acts faithfully on $V=V_1\oplus \dots\oplus V_r$, with each $V_j$ a (not necessarily faithful) $G$-subrepresentation, and let $G_j$ be the image of $G$ in $GL(V_j)$ for each $j$. Then
    \[
        \topdeg(G,V)\leq \sum_{j=1}^r \topdeg(G_j,V_j),
    \]
    and
    \[
        \Dspan(G,V)\leq \sum_{j=1}^r \Dspan(G_j,V_j).
    \]
    The hypothesis makes $G$ a subdirect product of $G_1,\dots,G_r$; if the product is direct, then also 
    \[
        \bfield(G,V)\leq\max_j \bfield(G_j,V_j).
    \]
\end{proposition}

\begin{proof}
    For $\topdeg$, the case where $r=2$ is immediate from \cite[Lemma~6]{kohls2014top} (which is stated for modular representations, but works in general), because the action of $G$ on each $V_j$ factors through $G_j$. The general case follows by induction on $r$.

    For $\Dspan$, we argue as follows (and essentially the same argument would also have worked for $\topdeg$). First, it is enough to show that the $\kk(V)^G$-span of the polynomials of degree up to $\sum_j \Dspan(G_j,V_j)$ contains $\kk[V]$, because $\kk[V]$ spans $\kk(V)$ over $\kk(V)^G$ (i.e.,  $\Xi$ as in \eqref{eq:Xi} is surjective). Second, we have
    \[
    \kk[V]\cong \kk[V_1] \cdots \kk[V_r]
    \]
    as graded algebras, i.e., $\kk[V]$ is $\kk$-spanned by elements of the form $f_1f_2\cdots f_r$ where each $f_j$ is a homogeneous element of $\kk[V_j]$. Third, each $f_j$ lies in the $\kk(V_j)^{G_j}$-span of $\kk[V_j]_{\leq \Dspan(G_j,V_j)}$. Since $\kk(V_j)^{G_j}\subseteq \kk(V)^G$ for each $j$, this means all of $\kk[V]$ lies in the $\kk(V)^G$-span of 
    \[
    \kk[V_1]_{\leq \Dspan(G_1,V_1)}\cdots \kk[V_r]_{\leq \Dspan(G_r,V_r)},
    \]
    which has elements of degree at most $\sum_{j=1}^r \Dspan(G_j,V_j)$. This completes the argument for $\Dspan$.

    The argument for $\bfield$ is similar. Note first that $\kk(V)^G\supseteq \kk(V_j)^{G_j}$ for each $j=1,\dots,r$. Then, the direct product hypothesis implies that
    \[
    \kk[V]^G = \kk[V_1]^{G_1} \cdots \kk[V_r]^{G_r},
    \]
    and thus
    \[
    \kk[V]^G \subseteq \kk(V_1)^{G_1} \cdots \kk(V_r)^{G_r},
    \]
    where the right side of this last displayed containment is the composite of fields. Since $\kk(V)^G$ is the fraction field of $\kk[V]^G$, it follows that $\kk(V)^G$ is the smallest field containing all the subfields $\kk(V_j)^{G_j}$. Then the union of generating sets for each $\kk(V_j)^{G_j}$ forms a generating set for $\kk(V)^G$. The stated inequality for $\bfield$ follows.
\end{proof}

\begin{remark}
    All three inequalities in Proposition~\ref{prop:subadditivity} are sharp. Equality is attained in all three, for example, when $G_1,\dots,G_r$ are all  cyclic, $G$ is their direct product, and $V_1,\dots,V_r$ are one-dimensional faithful representations over $\CC$ of $G_1,\dots,G_r$ respectively. In this case, we have $\bfield(G_j,V_j)=|G_j|$ and $\Dspan(G_j,V_j)=\topdeg(G_j,V_j)=|G_j|-1$ for each $j$, and meanwhile, $\bfield(G,V)=\max_j(|G_j|)$ and $\Dspan(G,V)=\topdeg(G,V)=\sum_j|G_j|-r$.
\end{remark}

\subsection{Characteristic-free bounds}

In this section we show that the point of view of the present work leads readily to the conclusion that $\Dspan$ is bounded by $|G|-1$, independent of the field characteristic. This provides a parallel story for $\Dspan$ to the fact that $\bfield$ is bounded above by $|G|$ in any characteristic \cite[Corollary~2.3]{fleischmann2007homomorphisms}, while this holds only in non-modular characteristic for the classical Noether number $\beta(G,V)$ \cite{noether, fleischmann2000noether, fogarty2001noether}; the latter in fact is not bounded in terms of $G$ alone if $\Char \kk$ divides $|G|$ \cite{richman1996invariants}. Indeed, the ring analogue of $\Dspan$, namely 
$\topdeg$,  satisfies
\[
\topdeg(G,V) \leq |G|-1,
\]
in characteristic zero \cite{schmid1991finite} or more generally the non-modular case, but is not bounded in terms of $G$ alone if $\Char \kk$ divides $|G|$ \cite{kohls2014top}.

As mentioned in the introduction, it also provides a mild refinement of a conclusion due to Koll\'ar and Tiep. The quantity $\Dreg$ is studied in \cite{kollar2024simple}. The main result is that any irreducible representation of $\kk G$ is both a submodule and a quotient of $\operatorname{Sym}^m(V)$ for some $1\leq m \leq |G|$. The proof actually shows that one can take $0\leq m \leq |G|-1$. It proceeds by finding the regular representation as a summand of 
\[
\operatorname{Sym}^{|G|-1}(V) \oplus \dots \oplus \operatorname{Sym}^{|G|-|Z|}(V),
\]
where $Z:=Z(G)$ is the center of $G$. Thus, replacing $V$ with $V^*$, this implies that \[
\Dreg \leq |G|-1,
\]
and this holds independent of the ground field.

We now demonstrate that the point of view adopted in the present work yields a very short proof that $\Dspan \leq |G| - 1$, independent of characteristic. By Proposition~\ref{prop:c-Dreg-Dspan-topdeg}, part~\ref{part:general-ineq}, this refines Koll\'ar and Tiep's conclusion that $\Dreg \leq |G| - 1$. 

\begin{theorem}\label{thm:refine-KP}
    If $G$ is a finite group, $\kk$ is a field, and $V$ is a finite-dimensional representation of $G$ over $\kk$, then
    \[
    \Dspan \leq |G| - 1.
    \]
\end{theorem}

\begin{proof}
    By definition, $\Dspan \leq d$ if and only if the restriction 
    \[
    \KK[X_1,\dots,X_n]_{\leq d} \rightarrow \kk(V)
    \]
    of $\Xi$ to $\KK \otimes \kk[V]_{\leq d}$ is surjective. Let us denote this restriction by $\Xi_d$ for brevity.
    
    Now $\kk(V)$ is a finite-dimensional algebra over $\KK$, of dimension $|G|$. The key idea is the following assertion:
    
    \begin{center}
        ($\star$) {\em If $\Xi_d$ is not surjective onto $\kk(V)$, then $\dim_\KK \im \Xi_{d+1} - \dim_\KK \im \Xi_d \geq 1$.}
    \end{center}
    This assertion ($\star$) implies the result by counting, as follows. If we set the (natural) convention that $\im \Xi_{-1}=0$, then ($\star$) also holds for $d=-1$ (because $\Xi_0$ has image $\KK$). If $d<\Dspan$, then $\Xi_d$ is not surjective, and it follows that for each of the $\Dspan + 1$ values of $d$ from $-1$ to $\Dspan - 1$, $\Xi_d$ fulfills the hypothesis of ($\star$). Thus, we have
    \begin{align*}
        |G| &= \dim_\KK \im \Xi_{\Dspan} \\
        &= \sum_{d=-1}^{\Dspan-1} \left(\dim_\KK \im \Xi_{d+1} - \dim_\KK \im \Xi_d\right)\\
        &\geq \Dspan + 1,
    \end{align*} 
    where the last inequality is by ($\star$). This yields the desired conclusion.
    
    We now prove ($\star$). Suppose, for a contradiction, that there exists $d\geq 0$ for which $\Xi_d$ is not surjective, but also for which $\dim_\KK \im \Xi_{d+1} = \dim_\KK \im \Xi_d$. Then the $\Xi$-image of $\KK[X_1,\dots,X_n]_{\leq d+1}$ is no bigger than the $\Xi$-image of $\KK[X_1,\dots,X_n]_d$, and it follows that the entire homogeneous component
    \[
    \KK[X_1,\dots,X_n]_{d+1}
    \]
    must lie in the generic orbit ideal $I$ (i.e., in the kernel of $\Xi$). But then every monomial of degree greater than $d+1$ must also lie in $I$, and it follows that for any $D>d$, $\im \Xi_D$ is no bigger than $\im \Xi_d$. Thus, 
    \[
    \im \Xi = \bigcup_{D\geq d} \im \Xi_D = \im \Xi_d,
    \]
    and $\Xi$ itself is not surjective. But $\Xi$ is surjective. This contradiction completes the proof.
\end{proof}

\begin{remark}
    The bound $\Dspan \leq |G|-1$ is sharp: equality is attained by any faithful scalar action of a cyclic group.
\end{remark}

We can give a generally tighter bound for permutation groups:

\begin{theorem}\label{thm:permutations}
    Suppose that $G$ acts on the $n$-dimensional $\kk$-vector space $V:=\kk^n$ by permuting the coordinates. Then $G$ acts on $\kk[V]=\kk[x_1,x_2,\dots,x_n]$ by permuting the variables. Let
     ${\mathcal O}_1,{\mathcal O}_2,\dots,{\mathcal O}_r$ be the orbits for the action of $G$ on $\{x_1,x_2,\dots,x_n\}$. We have 
  $$
    \Dspan(G,V)\leq \topdeg(G,V)\leq  \sum_{i=1}^r {n_i\choose 2}\leq {n\choose 2},
    $$
    where $n_i=|\mathcal O_i|$ for each $i=1,\dots,r$.
    \end{theorem}

\begin{proof}
    The group $G$ is contained in the larger permutation group $H=S_{n_1}\times S_{n_2}\times \cdots S_{n_r}$ with each factor acting separately on each $\mathcal O_i$ and thus on the sub-vector space $V_i\cong \kk^{n_i}$ of $V$ corresponding to the coordinates belonging to $\mathcal O_i$. For each $\mathcal O_i$, we have $\topdeg(S_{n_i}, V_i) = \binom{n_i}{2}$ by the main result of \cite{gobel1995computing}. Then
    \[
    \Dspan(G,V)\leq \topdeg(G,V) \leq \topdeg(H,V) \leq \sum_{i=1}^r \topdeg(S_{n_i},V_i) = \sum_{i=1}^r {n_i\choose 2},
    \]
    where the first inequality is from Proposition~\ref{prop:c-Dreg-Dspan-topdeg}, part~\ref{part:general-ineq}, the second is from Lemma~\ref{lem:monotonic-in-the-group}, and the third is from Proposition~\ref{prop:subadditivity}.
\end{proof}

We also give a trivial lower bound on $\Dspan$ obtained by dimension-counting:
\begin{proposition}\label{prop:dim-counting-lower-bound}
    If $G$ is a finite group, $\kk$ is a field, and $V$ is representation of $G$ over $\kk$ of finite dimension $n$, then
    \[
    \binom{\Dspan + n}{n}\geq \binom{\Dreg + n}{n} \geq |G|.
    \]
\end{proposition}

\begin{proof}
    Only the second inequality needs to be proven, as the first is immediate from Proposition~\ref{prop:c-Dreg-Dspan-topdeg}, part~\ref{part:general-ineq}. The second is almost as immediate: to contain the regular representation, $\kk[V]_{\leq d}$ must have at least its dimension as a $\kk$-vector space. So the desired inequality follows from the fact that
    \[
    \dim_\kk \kk[V]_{\leq d} = \binom{d+n}{n}.\qedhere
    \]
\end{proof}

\begin{remark}
    The second inequality in Proposition~\ref{prop:dim-counting-lower-bound} is almost never an equality because equality requires that the regular representation be exactly isomorphic to the direct sum of the first $\Dreg$ symmetric powers of $V^*$, rather than just contained in it. However, there are cases where this happens:
    \begin{enumerate}
        \item $V$ is one-dimensional (so $G$ is cyclic, acting by scalars). In this case, the lower bound in Proposition~\ref{prop:dim-counting-lower-bound} coincides with the upper bound in Theorem~\ref{thm:refine-KP}.

        \item The trivial representation of $G$ is a summand of its regular representation (this happens in non-modular characteristic) and $V$ is the complement. In this case, $\Dspan=\Dreg=1$ and $n=|G|-1$.
    \end{enumerate}
\end{remark}

\begin{corollary}\label{cor:explicit-lower-bound}
    Under the same hypotheses as Proposition~\ref{prop:dim-counting-lower-bound}, $\Dreg$ and $\Dspan$ satisfy
    \[
    \Dspan \geq \Dreg \geq \sqrt[n]{n!|G|}-\frac{n+1}{2}.
    \]
\end{corollary}

\begin{proof}
    One extracts this from Proposition~\ref{prop:dim-counting-lower-bound} by appying the estimate
    \[
    \binom{d+n}{n}=\frac{(d+n)\cdot\ldots\cdot(d+1)}{n\cdot\ldots \cdot 1}\leq \frac{\left(d+(n+1)/2)\right)^n}{n!}.\qedhere
    \]
\end{proof}

\begin{remark}
    The estimate in the corollary is only useful when $|G|$ is very large compared to $n$.
\end{remark}

\begin{remark}
    There is a similar lower bound for $\bfield$: by \cite[Theorem~3.2]{blum2024degree}, it satisfies $\bfield\geq \sqrt[n]{|G|}$.
\end{remark}

\section{An application}\label{sec:applications}

In this short section we apply Theorem~\ref{thm:main} to give  a bound on 
$\bfield$ for cyclic groups of prime order in terms of the character of $V$.

\begin{proposition}\label{prop:cyclotomic}
    Let $G=C_p$ be cyclic of odd prime order $p$ and let $V$ be a $\CC$-representation with character $\chi$. Let $\QQ(\chi)$ be the field generated over $\QQ$ by the values of $\chi$.  Then
    \[
    \bfield(G,V) \leq 2[\QQ(\chi):\QQ]+1.
    \]
\end{proposition}

\begin{proof}
    Note that $\kk:=\QQ(\chi)\subseteq \QQ(\zeta)$, where $\zeta$ is a primitive $p$th root of unity, thus it is a Galois extension of $\QQ$. We first claim that all the values of $\chi$ on the nonidentity elements of $G$ are Galois conjugate. We see this as follows.

    The automorphism group $\Aut(G)$ of $G$ is isomorphic to the unit group of $\ZZ/p\ZZ$; if $[a]\in (\ZZ/p\ZZ)^\times$ is an invertible residue class mod $p$,  where $a$ is an integer, then the corresponding automorphism $\phi_{[a]}$ of $G$ maps $g\mapsto g^a$ for all $g\in G$. Meanwhile, $(\ZZ/p\ZZ)^\times$ is also isomorphic to the Galois group of $\QQ(\zeta)/\QQ$, with the automorphism $\pi_{[a]}\in\Gal(\QQ(\zeta)/\QQ)$ that corresponds to $[a]$ sending $\zeta$ to $\zeta^a$. Then $\chi(\phi_{[a]}g) = \chi(g^a)=\pi_{[a]}(\chi(g))$, whereupon transitivity of the action of $\Aut(G)$ on the nonidentity elements in $G$ implies the claimed Galois conjugacy of the values of $\chi$ on these elements.

    Because $\chi(g)\in\kk$ for all $g$ (by definition of $\kk$) and the Galois conjugacy classes in $\kk$ have size at most $[\kk:\QQ]$, it follows that $\chi$ takes on at most $[\kk:\QQ]$ values on the nonidentity elements of $G$, and therefore at most $[\kk:\QQ]+1$ in total.

    Applying Brauer's theorem \cite[Theorem~$1^*$]{brauer1964note}, we obtain
    \begin{equation}\label{eq:from-brauer}
    \Dirr \leq [\kk:\QQ]+1 - 1 = [\kk:\QQ],
    \end{equation}
    where, as in the introduction and Section~\ref{sec:Dirr-Dreg-Dspan-topdeg}, $\Dirr$ is the lowest $d$ such that every irreducible representation of $G$ occurs as a subrepresentation of one the first $\Dirr$ tensor powers of $V$. Therefore,
    \[
    \bfield\leq 2\Dspan+1 = 2\Dirr+1 \leq 2[\kk:\QQ]+1,
    \]
    where the first inequality is from Theorem~\ref{thm:main}, the (middle) equality is from Proposition~\ref{prop:c-Dreg-Dspan-topdeg}, part~\ref{part:abelian-ineq}, and the final inequality is from \eqref{eq:from-brauer}.
\end{proof}

\begin{example}\label{ex:five}
Let $V$ be a representation of $G=C_p$ defined over a (real or imaginary) quadratic extension of $\QQ$, for example the representation of dimension $(p-1)/2$ whose character on a fixed generator is the sum $\sum_{a\in U}\zeta^a$, where $\zeta$ is a fixed primitive $p$th root of unity and $U\subset (\ZZ/p\ZZ)^\times$ is the set of quadratic residues mod $p$. Then Proposition~\ref{prop:cyclotomic} gives us that $\bfield(G,V)\leq 5$.
\end{example}

\section*{Acknowledgements}

We thank Victor Reiner for calling our attention to \cite[Theorem~$1^*$]{brauer1964note}, Fabian Reimers and M\"ufit Sezer for helpful feedback, and two anonymous referees for thoughtful comments which greatly improved the paper. BBS visited HD at Northeastern University in October 2024, and gratefully acknowledges the hospitality of the Mathematics Department. BBS was partially supported by Soledad Villar's NSF CAREER 2339682. HD was partially supported by NSF DMS 2147769 and a Simons Fellowship.

\bibliographystyle{alpha}
\bibliography{bib}

\end{document}